\documentclass[11pt]{article}

\usepackage[a4paper,scale=1.0]{geometry}
\usepackage{amsmath}
\usepackage{amssymb}
\usepackage{amsthm}
\usepackage{bm} 
\usepackage{booktabs}
\usepackage{color}
\usepackage{commath}
\usepackage{doi}
\usepackage[numbers,sort&compress]{natbib}
\usepackage{enumitem}
\usepackage{float}
\usepackage{fullpage}
\usepackage{graphicx}
\usepackage{hyperref}
\usepackage{subcaption}
\usepackage{thmtools}
\usepackage{tocloft}
\usepackage{ulem}
\usepackage{url}

\usepackage{array}
\usepackage{ragged2e}
\usepackage[font=small,labelfont=bf]{caption}
\setenumerate{font=\itshape}

\newcommand{\UP}[3][]{\makebox[0pt]{\smash{\raisebox{1.4em}{$\phantom{#3}$#1\rotatebox{45}{\footnotesize $#2$}}}}{#3}}
\newcommand{\UPn}[3][]{\makebox[0pt]{\smash{\raisebox{1.9em}{$\phantom{#3}$#1\rotatebox{45}{\footnotesize $#2$}}}}{#3}}
\newcommand{\RG}[2][]{\makebox[0pt]{\hspace{3em}#1\rotatebox{15}{\footnotesize $#2$}}}

\newcommand{\HRule}[1][\medskipamount]{\par
  \vspace*{\dimexpr-\parskip-\baselineskip+#1}
  \noindent\rule{\linewidth}{0.2mm}
  \vspace*{\dimexpr-\parskip-.5\baselineskip+#1}
}

\hypersetup{colorlinks,linktocpage,citecolor=red,filecolor=red,linkcolor=red,urlcolor=red}

\newcommand{\es}{\varnothing}
\newcommand{\mI}{\mathcal{I}}
\newcommand{\mJ}{\mathcal{J}}
\newcommand{\mA}{\mathcal{A}}

\newcommand{\PNn}{P_{N,n}}
\newcommand{\HNn}{H_{N,n}}

\newcommand{\conv}{\mathop{\mathrm{conv}}}

\newcommand{\RR}{{\mathbb{R}}}

\definecolor{DAcolor}{rgb}{0.1,0.1,1}

\definecolor{SHCcolor}{rgb}{0.9,0.1,0.1}

\newcommand{\ignore}[1]{}

\DeclareMathOperator{\area}{area}

\DeclareMathOperator{\ILP}{ILP}

\theoremstyle{plain}
\newtheorem{theorem}{Theorem}
\newtheorem{corollary}{Corollary}
\newtheorem{lemma}{Lemma}
\newtheorem{proposition}{Proposition}

\theoremstyle{definition}
\newtheorem{definition}{Definition}

\theoremstyle{remark}

\newtheorem{problem}{Problem}

\newcounter{innerlist}

\let\cite\citep

\newlength{\normalparindent}
\newlength{\normalparskip}
\numberwithin{equation}{section}

\newcommand{\mi}[1]{q_{#1}}
\newcommand{\comp}{\complement}
\newcommand{\sminus}{\setminus}
\newcommand{\extn}{[n;N]}

\newcommand{\progname}{\textsf}

\newcommand{\mait}{\progname{mai20}}
\newcommand{\lrslib}{\progname{lrslib}}

\newcommand{\mplrs}{\progname{mplrs}}
\newcommand{\normaliz}{\progname{Normaliz}}
\newcommand{\cplex}{\progname{CPLEX}}
\newcommand{\glpsol}{\progname{glpsol}}
\newcommand{\gurobi}{\progname{Gurobi}}

\newcommand{\tool}[1]{\vspace{.25cm}\phantomsection\centerline{\large \textbf{#1}}\vspace{.3cm}}

\RequirePackage{fix-cm}

\begin{document}

\title{
On the foundations and extremal structure \\ of the holographic entropy cone 
\thanks{DA was supported by JSPS Kakenhi Grants
16H02785, 
18H05291
and
20H00579.\\
\indent\;\: SHC was supported by NSF Grant PHY-1801805
and the University of California, Santa Barbara.
}
}

\author{David Avis \\ School of Informatics, Kyoto University, Kyoto, Japan and
   \\ School of Computer Science,
    McGill University, Montr{\'e}al, Qu{\'e}bec, Canada
\and Sergio Hern{\'a}ndez-Cuenca \\
Center for Theoretical Physics, Massachusetts Institute of Technology, Cambridge, MA, USA}

\providecommand{\keywords}[1]{\noindent\textbf{\textit{Keywords:\,}} #1}

\maketitle

\begin{abstract}
The holographic entropy cone (HEC) is a polyhedral cone first introduced in the
study of a class of quantum entropy inequalities.
It admits a graph-theoretic description in terms of
minimum cuts in weighted graphs, a characterization which naturally
generalizes the cut function for complete graphs. Unfortunately,
no complete facet or extreme-ray representation of
the HEC is known. In this work, starting from a purely graph-theoretic perspective,
we develop a theoretical and computational foundation for the HEC.
The paper is self-contained, giving new proofs of known results and proving several
new results as well. 
These are also used to develop two systematic 
approaches for finding the facets and extreme rays of the HEC, which we illustrate
by recomputing the HEC on $5$ terminals and improving its graph description. 
We also report on some partial results for 6 terminals. 
Some interesting open problems are stated throughout.
\vspace{4pt}

\end{abstract}

\tableofcontents

\section{Introduction}
\label{intro}

The holographic entropy cone (HEC) has its origins in quantum physics 
in the work of Bao et al. \cite{Bao:2015bfa} as described briefly in 
Appendix \ref{physics}. The HEC is a family of polyhedral cones $H_n, n \ge  1.$
A crucial result of their paper 
is a graph-theoretic characterization
in terms of minimum cuts in a complete graph, which is a natural
generalization of the well-studied cone of cut functions. 
This allows us to study the HEC without any reference to the underlying quantum physics setting. Apart from its relationship to cut functions, the HEC does not appear to be
related to other known polyhedral objects.
Our main focus is on the extremal structure of the HEC.
At present no compact representation of either the extreme rays or the facets 
of $H_n$ is known and a complete explicit description is only
known up to $n=5$, see \cite{Bao:2015bfa,Cuenca:2019uzx}.
The main motivation for studying the extremal structure of the HEC is the characterization of its facets, 
which physically correspond to entropy inequalities that strongly constrain the entanglement patterns of 
quantum states encoding higher-dimensional spacetimes as their quantum gravity duals in 
holography \cite{Bao:2015bfa,Chen:2021lnq} -- see Appendix \ref{physics} for more details

This paper is structured as follows. Firstly, 
in Section \ref{defs} we give a formal definition of $H_n$ and some
basic structural results that will be needed throughout the paper.
These include new proofs that it is full-dimensional and 
polyhedral. In proving the latter result, using antichains in a lattice, 
we obtain a tighter bound on
the size of the complete graph needed to realize all extreme rays of $H_n$.
We then review some basic results on the $H$- and $V$-representations
of cones and study $H_2$ relating it to the cone of cut functions.
In Section \ref{rays} we discuss the extreme rays of $H_n$ and
describe some related cones that lead to methods to compute them. 
This gives a simple proof that the $H_n$ is a rational cone.
It also allows us to give a description of $H_3$.
Following that we give a general zero-lifting result for extreme rays.
In Section
\ref{facets} we describe valid inequalities and facets.
We begin by reviewing the proof-by-contraction method
that is used for proving validity of inequalities. In the proof we again use 
antichains, obtaining a reduction in the complexity of the original method. 
This is followed by a discussion of zero-lifting of valid inequalities and facets.
In Section \ref{ilp} we describe integer programs that can be used to test
membership in $H_n$ and prove non-validity of inequalities defined over it.
Many of the results of the paper are combined in Section \ref{H5}, which describes 
two methods to derive complete facet and extreme-ray descriptions of $H_n$
and illustrate these on computations of $H_5$.
There are a lot of interesting open problems related to the HEC, and some of these
are stated throughout the paper and in the conclusion.
Supplemental material, including
input and output files, integer linear programs and
C code for various functions mentioned,
is available online\footnote{\url{http://cgm.cs.mcgill.ca/~avis/doc/HEC/HEC.html}}.

\section{Definitions and basic results}
\label{defs}
For any positive integers $k$ and $N$, let $[k]=\{1,2,\dots,k \}$, and let $K_N$ 
denote the undirected complete graph on the vertex
set $[N]$. The edge set $E_N$ consists of all edges 
$e=(i,j)$ between vertices $i,j\in [N]$ for every pair $1 \le i < j \le N$. 
A weight map $w : E_N \to \RR_{\ge 0}$ is introduced to assign a nonnegative 
weight $w(e)$ to every $e\in E_N$. Any subset $W \subseteq [N]$ defines a
\textit{cut} $C(W)$ as the set of all edges $(i,j)$ with $i \in W$ and $j
\notin W$.
Since both $W$ and its complement define the same cut, we will normally consider
cuts where $W \subseteq [N-1]$, and generally exclude
the empty cut.
We denote by $S_W=\norm{C(W)}$ the total weight of the cut $C(W)$, which is
the sum of the weights of all the edges in $C(W)$. Letting $n=N-1$, consider 
the \textit{$S$-vector} of length ${2^n-1}$ with entries indexed by cardinality 
and then lexicographically by the non-empty subsets of $[n]$, 
\begin{equation}
\label{Svector} 
    S=(S_{1},S_{2},\dots,S_{n},S_{12},\dots,S_{12\dots n}),
\end{equation} 
where juxtaposition is a shorthand for the corresponding set of
integers. The convex hull of the set of all $S$ vectors for a given $N$ forms a
cone in $\RR^{{2^n-1}}$. In fact, this cone is polyhedral, 
its facets are the
subadditivity inequalities and the submodular inequalities are valid for it,
as established independently by Tomizawa and
Fujishige (see Section 3.6 of \cite{Fu91}) and Cunningham \cite{Cu85}. 
When the vector $S$ is
expressed as a function of $W$ it is known as the \textit{cut function}.

The HEC is a generalization of the cone defined by the cut function. We follow
\cite{Bao:2015bfa}, but adapt its notation and terminology 
considerably. Instead of setting $n=N-1$, we fix some integer 
$n\ge 2$ and consider $K_N$ for all $N>n$. In any such graph, we call the 
vertices $[n]$ \textit{terminals} (cf. boundary regions in holography).
The vertex $N$ is called the \textit{purifying vertex} in the physics literature,
but we will simply call it the \textit{sink} here.
Oftentimes, these will be combined into $\extn = [n]\cup\{N\}$ and collectively 
referred to as \textit{extended terminals}.
The other vertices, if any, are called \textit{bulk} vertices (cf.
the bulk spacetime).

Let $I$ be a non-empty subset of terminals, i.e. $\es \ne I \subseteq [n]$. We
extend the definition of $S$ above to this new setting.
For any $N > n$ and weight map $w$ defined on $K_N$, we introduce a construct which
captures all the basic properties conveyed by the RT formula in \eqref{eq:RT}.
In particular, let
\begin{equation}
\label{Sdef} 
    S_I = \min_{I=W \cap [n]} \; \norm{C(W)},
\end{equation} 
where the minimization is over all $W\subseteq [N-1]$. This says
that $S_I$ takes the minimum weight over all cuts which contain precisely
the terminals in $I$ and some (possibly empty) subset of the bulk vertices.
Note that when $n=N-1$ we are minimizing over the single
subset $W=I$ and the definition is equivalent to the one given earlier. In
graph theory terms, $S_I$ is just the capacity of the minimum-weight cut 
or \textit{min-cut} in $K_N$ separating $I$ from $\extn\sminus I$.
By the duality of cuts and flows, an equivalent definition is
to let $S_I$ be the value of the maximum flow between multiple-sources
$I$ and multiple-sinks $\extn\sminus I$ in $K_N$.
The max flow problems are structurally different for each $I$ but nevertheless
give an efficient method of computing the $S_I$. 

We form an $S$-vector from \eqref{Sdef} 
of the form of \eqref{Svector} as we did previously,
and say that $w$ \textit{realizes} $S$ in $K_N$
or, more compactly, that $(S,w)$ is a \textit{valid pair}.
\begin{definition}\label{def:hec}
    The \textit{holographic entropy cone} on $n$ terminals is defined as
    \begin{equation}\label{eq:hecdef}
        H_n = \{S \in \RR^{2^n-1} \;:\; 
    \textit{$(S,w)$ is a valid pair for some $N$ and $w$} \}.
    \end{equation}
\end{definition}
Examples of the facet defining inequalities and extreme rays of $H_n$ for small $n$
are given in Appendices \ref{app:facets} and \ref{app:rays} respectively.

It follows from \eqref{Sdef} that for any $\lambda > 0$, 
$(S,w)$ is a valid pair for $H_n$
if and only if $(\lambda S, \lambda w)$ is, so $H_n$ is a cone.
In fact it is full-dimensional. 
The proof employs $S$-vectors arising from $K_{n+2}$ with
all edges of zero weight except possibly edges $(i,n+1)$ for $i\in\extn$.
We call these \textit{star graphs} and exhibit a family of $2^n -1$ of
them giving linearly independent $S$-vectors.

\begin{proposition}\label{prop:full}
    $H_n$ is a cone of dimension $2^n-1$.
\end{proposition}

\begin{proof}
For each $\es\ne J \subseteq [n]$, define a weighted star graph where the nonzero 
edge weights are\footnote{This class of star graphs were inspired by a construction 
of \cite{Hubeny:2018ijt}.}
\begin{equation}
\label{eq:star}
    w_i= 1,\quad\forall~i \in J\qquad\text{and}\qquad w_N=
\begin{cases}
    1 \qquad &\text{if \;$\abs{J} = 1$},\\
    \abs{J}-1 \qquad &\text{otherwise}.
\end{cases}
\end{equation}
For every $n \ge 2$, their respective $S$-vectors $S^J$ can be easily seen to be given by
\begin{equation}
\label{eq:li}
    S_I^J= \abs{I \cap J} -\delta (I,J), \qquad \delta(I,J) = 
\begin{cases}
    1 \qquad &\text{if \;$\abs{J} \ge 2$ and $I \supseteq J$},\\
    0 \qquad &\text{otherwise}.
\end{cases}
\end{equation}
Using them as row vectors, we construct square matrices $A^n$. 
For example,
\vspace{.5cm}\begin{equation}
\begin{aligned}
    A^2 &=
    \arraycolsep=0.34em\renewcommand{\arraystretch}{1.1}
    \left[
    \small
    \begin{array}{c|c|c}
        \UP{1}{1}&\UP{12}{1}&\UP{2}{0}\RG{1}\\\hline
        1&1&1\RG[~]{12}\\\hline
        0&1&1\RG{2}
    \end{array}
    \right]\quad
    \\~\\
    A^3 &= 
    \arraycolsep=0.34em\renewcommand{\arraystretch}{1.1}
    \left[
    \small
    \begin{array}{ccc|ccc|c}
        \UP{1}{1}&\UP{2}{0}&\UP{12}{1}&\UP{13}{1}&\UP{23}{0}&\UP{123}{1}&\UP{3}{0}\RG{1}\\
        0&1&1&0&1&1&0\RG{2}\\
        1&1&1&1&1&1&0\RG[~]{12}\\\hline
        \textcolor{red}{1}&0&1&\textcolor{red}{1}&1&1&1\RG[~]{13}\\
        \textcolor{blue}{0}&\textcolor{red}{1}&1&\textcolor{blue}{1}&\textcolor{red}{1}&1&1\RG[~]{23}\\
        \textcolor{blue}{1}&\textcolor{blue}{1}&\textcolor{red}{2}&\textcolor{blue}{2}&
        \textcolor{blue}{2}&\textcolor{red}{2}&1\RG[~~]{123}\\\hline
        0&0&0&1&1&1&1\RG{3}\\
    \end{array}
    \right]\quad
\end{aligned}
\quad\cdots\quad
\arraycolsep=0.4em\renewcommand{\arraystretch}{1.25}
A^{n+1} = 
    \left[
        \begin{array}{ccc}
        \UPn[]{I}{B^n\vspace*{10cm}}&\UPn[\quad~]{I\cup\{n+1\}}{C^n}&
        \UPn[\quad]{\{n+1\}}{c}\RG[~~]{J}\\
        D^n&E^n&\vdots\RG[\qquad\qquad~~]{J\cup\{n+1\}}\\
        d&\dots&\cdot\RG[\qquad~~~~]{\{n+1\}}
        \end{array}
    \right]\qquad\qquad\quad
\label{eq:A3}
\end{equation}
where the general sketch partitions $A^{n+1}$ into four square matrices $B^n$, $C^n$, $D^n$ 
and $E^n$, of size $2^n-1$, a final column $c$, and a final row $d$.
Note that
the rows and columns have been permuted from their
usual ordering for subsets of $[n+1]$. Here, labels $\es\ne I\subseteq[n]$ go first, 
then those of the form $I\cup\{n+1\}$, and $\{n+1\}$ last (cf. the block forms in \eqref{eq:A3}).
We prove by induction on $n$ that $\det(A^{n})=(-1)^{n+1}$.
This is immediate for $n=2$.
Matrix $B^n$ in $A^{n+1}$ is just $A^n$ reordered as described above.
Since we perform the same reordering for rows as for columns the determinant sign is unchanged,
so by the induction hypothesis $\det(B^n)=(-1)^{n+1}$.
As the rows of $B^n$ and $C^n$ are indexed by $J\not\ni n+1$, we have $C^n=B^n$.
Additionally, one easily verifies that in $c$ the first $2^n-1$ entries are $0$ and the 
rest are $1$. Row $d$ has the same pattern.

We now make a comparison between entries in column 
$I\not\ni n+1$ of $D^n$ and column $I \cup \{n+1 \}$ of $E^n$. 
Consider the diagonal elements of each. For row $J$, in $D^n$ we have column 
$I=J \setminus \{n+1\}$ and so $S_I^J= \abs{I} = \abs{J}-1$. In $E^n$ the column 
label is also $J$ and since $\abs{J} \ge 2$ we have
$\delta(I,J)=1$ and so $S_I^J= \abs{J}-1$. Hence the diagonals are identical.
Now consider the elements below them. For $D^n$ each row index $J$ contains 
$n+1$ but none of its column indices do, so $S_I^J= \abs{I \cap J}$. In $E^n$ the
same applies but the intersection now includes $n+1$, so the corresponding entry 
is always bigger by one. These facts are illustrated by the coloured entries in $A^3$.

We now subtract the first $2^n-1$ columns of $A^{n+1}$ from the next $2^n-1$ columns, 
then subtract $c$ from each of these columns also, obtaining
\begin{equation}
\arraycolsep=0.34em\renewcommand{\arraystretch}{1.1}
\label{eq:matrices}
    A^{n+1} = 
    \left[
        \begin{array}{ccc}
        B^n&B^n&\bf{0}\\
        D^n&E^n&\bf{1}\\
        \bf{0}&\bf{1}&1
        \end{array}
    \right]
    \qquad\longrightarrow\qquad
    \tilde{A}^{n+1} = 
    \left[
        \begin{array}{ccc}
        B^n&\bf{0}&\bf{0}\\
        D^n&E^n-D^n-\bf{1}&\bf{1}\\
        \bf{0}&\bf{0}&1
        \end{array}
    \right].
\end{equation}
Here $\bf{0}$ and $\bf{1}$ respectively denote all-$0$ or all-$1$ matrices of suitable size.
In the resulting $\tilde{A}^{n+1}$, notice that $E^n-D^n-\bf{1}$ is an upper triangular 
matrix with all diagonal elements equal to $-1$.
Recalling that $\det(B^n)=(-1)^{n+1}$, we have 
$\det(A^{n+1})=\det(\tilde{A}^{n+1})=(-1)^{n+2}$, as desired.
\end{proof}

We will show in the following sections that $H_n$ is also convex, polyhedral and rational.
One important basic property the $S$-vectors do not possess is monotonicity, as can be seen
by examples in Appendix \ref{app:rays}.
 
\subsection{Polyhedrality of the HEC}\label{ssec:poly}

The polyhedrality of the HEC was established by Bao et al. \cite{Bao:2015bfa}.
We give a proof of this crucial result here following similar lines to the original
proof but obtain a tighter result due to our use of antichains.
In general, there may be more than one min-cut $W$ for each $\es\ne I\subseteq [n]$ 
achieving the minimum in \eqref{Sdef}. Among these, let $W_I$ denote one which is minimal under set inclusion.
We call $W_I$ a \textit{minimal min-cut} for $I$ and have $S_I = \norm{C(W_I)}$.
The following basic theorem shows that these are unique and builds on results from
Lemma $6$ of \cite{Nezami:2016zni}, and Theorems $3.1$ and $3.2$ of \cite{Bao:2020mqq}.
\begin{theorem}
\label{basic}
For positive integers $n<N$,
consider a weighted complete graph $K_N$ with terminal set $[n]$.
Let $W_I$ and $W_J$ be minimal min-cuts for $\es\ne I, J \subseteq [n]$.
Then:
\begin{enumerate}[label=(\alph*),leftmargin=2\parindent]
    \item \label{thm:unique} Each $I \subseteq [n]$ has a unique minimal min-cut $W_I$.
    \item \label{thm:nest} \qquad~~~~~~$I \subseteq J \quad \iff \quad W_I \subseteq W_J$.
    \item \label{thm:dist} \qquad\;$ I  \cap J = \es \quad \iff \quad W_I \cap W_J = \es $.
    \item \label{thm:bound} If $ m=\abs{\cup_{I \subseteq [n]} W_I}$, then all minimal
    min-cuts can be represented in a weighted $K_{m+1}$.
\end{enumerate}
\end{theorem}

\begin{proof}~
\begin{enumerate}[label=(\alph*),leftmargin=2\parindent]
    \item 
    Suppose $W$ and $W'$ are minimal min-cuts for $I$. 
    Submodularity of the cut function gives
    \begin{equation}\label{eq:ssa}
        \norm{C(W)} + \norm{C(W')} \ge \norm{C(W \cup W')} + \norm{C(W \cap W')}.
    \end{equation}
    Clearly, $W \cup W'$ and $W \cap W'$ are cuts for $I$. Since $W$ and $W'$ are 
    additionally min-cuts,
    \begin{equation}
        \norm{C(W \cup W')} \ge \norm{C(W)},
        \qquad \norm{C(W \cap W')} \ge \norm{C(W')},
    \end{equation}
    thereby turning all inequalities above into equations.
    Hence $W \cap W'$ is a min-cut and, as an intersection of minimal ones, minimal as well.
    It must thus be the case that $W = W' = W_I$.
    
    \item 
    First assume that $W_I \subseteq W_J$.
    Since $W_I$ and $W_J$ are cuts for $I$ and $J$ respectively, 
    we have $W_I \cap [n] = I$ and $W_J \cap [n] = J$.
    As $W_I \subseteq W_J$, we have $W_I \cap [n] \subseteq W_J \cap [n]$.
    Hence $I \subseteq J$.

    Now assume that $I \subseteq J$. Again, as min-cuts, 
    $W_I \cap [n] = I$ and $W_J \cap [n]=J$, and therefore 
    $(W_I \cap W_J) \cap [n] = I$ and $(W_I \cup W_J) \cap [n] = J$. This 
    means $W_I \cap W_J$ and $W_I \cup W_J$ are respectively cuts for $I$ 
    and $J$. Then submodularity and minimality, applied to 
    $W=W_J$ and $W'=W_I$ as in the proof of \ref{thm:unique} above, imply 
    $W_I \cap W_J = W_I$, which proves the claim.
    
    \item 
    By the definitions, $W_I \cap W_J = \es $ implies that $ I \cap J = \es$.

    For the converse, suppose that $ I \cap J = \es$ and 
    that there exists a vertex $x \in W_I \cap W_J $. Let $a, b$ and $c$ be the total weight of 
    edges from $x$ to, respectively, 
    $W_I \sminus W_J$, $W_J \sminus W_I$ and $[N] \sminus 
    (W_I \cup W_J)$. 
    Since $W_I$ is a min-cut, $a>b+c$, for otherwise we could remove
    $x$ from $W_I$ without increasing the weight of the cut.
    Similarly, by considering $W_J$, we have $b>a+c$.
    As edge weights are nonnegative, this gives the desired contradiction.

    \item 
    Firstly, we renumber the vertices $n+1,\dots,N$ in $K_N$ so that vertices 
    $[m]$ cover all of the vertices in the union of the $W_I$. 
    In $K_{m+1}$ we will let $m+1$ take the role of the sink $N$ and adjust 
    weights as follows. We leave the edge weights unchanged between edges with 
    both endpoints in $[m]$. For $i \le m$ we give edge $(i,m+1)$ the weight 
    corresponding to the sum of the weights of all edges $(i,j)$ with 
    $j=m+1,\dots,N$. It is easy to verify that the weights of the min-cuts 
    $W_I$ are preserved: if a smaller weight cut for a terminal set $I$
    existed in $K_{m+1}$, then it could be reproduced in the original $K_N$, 
    a contradiction.
\end{enumerate}
\vspace{-10pt}
\end{proof}

Unfortunately, part \ref{thm:nest} above does not generalize to the intersection 
of three or more sets. A simple example is given by the $K_{5}$ star graph with
unit weights for the $3$ terminal edges and zero for the sink edge. 
In particular, the intersection of the three pairs of terminals is of course empty,
but the intersection of their minimal min-cuts is not as it contains the bulk vertex.

Each $S$-vector is realized in $K_N$ for some $N$, and we are interested
in the smallest such $N$. More generally, for a given $n$, is there a smallest
integer $m(n)$ such that all $S$-vectors on $[n]$ can be realized in $K_{m(n)}$?
The answer is yes and this was proved by Bao et al. \cite{Bao:2015bfa} (Lemma $6$) 
who obtained $m(n) \le 2^{2^n-1}$. A tighter bound can be obtained from 
Theorem \ref{basic} as follows.

Let $Bool_n$ denote the \textit{Boolean lattice} of all subsets of
$[n]$ ordered under inclusion.
A family of subsets of $[n]$, $\mI\subseteq Bool_n$, is an
\textit{upper set} if for each $I \in \mI$ and $J \subseteq [n]$ that contains
$I$ we have $J \in \mI$.
For $\mJ\subseteq Bool_n$,
we call $\mJ$ \textit{pairwise intersecting} if each pair of its constituent
subsets has a non-empty intersection. If $\mJ$ is the empty set or
consists of a singleton, we consider $\mJ$ to be pairwise intersecting.
An \textit{antichain} $\mA \subseteq Bool_n$ is a collection
of subsets of $[n]$ which are pairwise incomparable, i.e. none of them is
contained in any of the others. Notice that the minimal elements of any upper set
form an antichain $\mA$ and that $\mA$ is pairwise intersecting if and
only if its upper set is.
Let $M(n)$ denote the number of pairwise interesecting antichains
$\mA$ in $Bool_n$.
We can use this value to bound $m(n)$ as follows:

\begin{corollary}
\label{cor:bound}
For $n \ge 2$, every $S$-vector for $n$ terminals can be realized in $K_{m(n)}$, where 
\begin{equation}
    m(n) \le M(n).
\end{equation}
\end{corollary}

\begin{proof}
We first sketch the argument in \cite{Bao:2015bfa} for their upper bound on 
$m(n)$. Suppose a given $S$-vector on $n$ terminals can be realized in a weighted 
$K_N$, for some given $N$. For $I \subseteq [n]$, $W_I$ partitions the vertex set 
$[N]$ of $K_N$ into two subsets. If we intersect these by $W_J$, for some 
$I \ne J \subseteq [n]$, we get $4$ subsets, some possibly empty. After repeating 
for all $2^n-1$ non-empty subsets of $[n]$ we obtain a partition of $[N]$ into 
$2^{2^n-1}$ subsets, many of which may be empty.
However, in each of the non-empty subsets, the vertices of $K_N$ may be merged 
into a single vertex by combining edge weights (cf. Theorem \ref{basic}\ref{thm:bound}). 
This new complete graph has at most $2^{2^n-1}$ vertices and realizes the same min-cut weights 
as before, giving their result.

To improve this bound we use Theorem \ref{basic}. For a set $W\subseteq[N]$, 
denote its complement by $W^\comp = [N] \sminus W$. Any atom in the partition 
just described is formed by splitting the non-empty subsets of $[n]$
into two disjoint, spanning families $\mI$ and $\mJ$, and taking the intersection
\begin{equation}
\label{Wpart}
    \bigcap_{I \in \mI} W_{I}~~ \cap ~~\bigcap_{J \in \mJ} W_J^\comp.
\end{equation}
Suppose this intersection is non-empty. 
Theorem \ref{basic}\ref{thm:dist} implies that $\mI$ is pairwise intersecting,
for otherwise the left intersection in \eqref{Wpart} is empty. In particular, 
this implies that both a subset and its complement cannot be in $\mI$.
In addition, one can show that $\mI$ must either be empty or an upper set in 
$Bool_n$ as follows. If $\mI=\es$, then \eqref{Wpart} is in fact never empty 
because it will always contain vertex $N$. 
As for $\mI\ne\es$, consider two subsets $I \subset K \subseteq [n]$ and
suppose $I \in \mI$ is non-empty. We have by Theorem \ref{basic}\ref{thm:nest}
that $W_I \subseteq W_K$ and so
$W_I \cap W_K^\comp = \es$, implying that if $K \in \mJ$, then \eqref{Wpart} is 
empty. As a result, either $\mI=\es$ or $\mI$ must be a pairwise intersecting 
upper set in $Bool_n$, with $\mJ$ containing all other non-empty subsets of 
$[n]$.

Because empty atoms from \eqref{Wpart} do not contribute to min-cut weights, it 
follows  that when considering $S$-vectors we need only be concerned with pairwise 
intersecting upper sets $\mI$ in $Bool_n$ and $\mI=\es$. As described above, 
the upper sets can be equivalently enumerated as the number of pairwise 
intersecting antichains in $Bool_n$. Since $M(n)$ counts their number,
we conclude that all $S$-vectors with $n$ terminals can be realized in $K_{M(n)}$.
\end{proof}

We have the following reasonably tight asymptotic bounds on $M(n)$.
Let $\bar{M}(n)$ be the total number of antichains in
$Bool_n$. Then, 
\begin{equation}\label{eq:mnbo}
    { \binom{n}{\lfloor \frac{n}{2} \rfloor +1}} ~<~ \log_2 M(n) ~<~
    \log_2 \bar M(n) ~\sim~ \binom{n}{\lfloor \frac{n}{2} \rfloor} ~\sim~
    \frac{2^{n+1}}{\sqrt{2 \pi n}}.
\end{equation}
The asymptotic upper bound on $\bar{M}(n)$ is due to Kleitman and Markowsky \cite{KM75}.
The lower bound can be obtained by considering all subsets of 
$[n]$ of size $\lfloor n/2 \rfloor + 1$. Each  pair of such subsets intersects
and none can properly contain another. So any collection of these subsets forms
an intersecting antichain.
While we do not know of tighter asymptotic bounds for $M(n)$, exact values
are known \cite{brouwer2013counting}\footnote{$M(n)$ is entry $n+1$ in 
Proposition $1.2$ of \cite{brouwer2013counting}.} for 
$1 \le n \le 8$:
\begin{equation}
    2,~~4,~~12,~~81,~~2646,~~1422564,~~229809982112,~~423295099074735261880.
\end{equation}
However, it seems that $M(n)$ is a very poor upper bound on $m(n)$.
For example, data for $1 \le n \le 4$ shows that $m(n)=2,3,5,6$ -- see Section \ref{ilp} 
for more details.

\begin{problem}
\label{prob:boundm}
Find tighter bounds on $m(n)$. In particular, does $\log_2 m(n)$ admit an upper 
bound that is polynomial in $n$?
\end{problem}

Definition \ref{def:hec} suggests the following family of cones, which are useful
in proving the polyhedrality of $H_n$. For any pair of integers $n$ and $N$ 
such that $2\le n+1 \le N$, consider
\begin{equation}
\label{HNndef}
    H_{N,n} = \conv~\{S \in \RR^{{2^n-1}} \;:\;
    \textit{nonnegative weighted $K_N$ s.t. $S$ satisfies \eqref{Sdef}} \}.
\end{equation}
This is a generalization of the cone defined by the
cut functions, which corresponds to the specific case $n=N-1$. 
Without the
convex hull operator in \eqref{HNndef}, $H_{N,n}$ would not be convex in general,
as shown by example in Appendix \ref{newS}.
An important property of $H_{N,n}$ is that it is naturally invariant under 
the action of the symmetric group
$Sym_{n}$ which permutes the elements of the set $[n]$. In fact, $H_{N,n}$ enjoys
a larger symmetry group: it is symmetric under permutations of vertices in the 
extended terminal set $\extn$. The permutations of $K_N$ under $Sym_{n+1}$ yield 
$S$-vectors \eqref{Svector} which are related by the simple fact that 
$C(W)=C(W^\comp)$. Similarly, in an undirected graph any 
min-cut is insensitive to the exchange of its sources and sinks. When talking 
about symmetries under $Sym_{n+1}$, it is thus convenient to identify 
$S_{\extn \sminus I} = S_{I}$. 

It is easy to see that in general $H_{N,n} \subseteq H_{N+1,n}$, since an 
additional bulk vertex can always be added to $K_N$ with all edges containing 
it of weight zero. 
We are now able to prove that $H_n$ is a convex polyhedral cone.
Our bound is tighter than the original proof given in \cite{Bao:2015bfa} due
to our use of antichains.
\begin{corollary}
\label{cor:convex}
For any $n \ge 1$, $H_n$ 
is a convex, polyhedral cone given by $ H_{n} = H_{m(n),n}$.
\end{corollary}
\begin{proof}
Firstly, suppose $S \in H_n$. Then for some $N$ and weight set $w$ the pair 
$(S,w)$ satisfies \eqref{Sdef} and so $S \in H_{N,n} \subseteq H_n$.
Conversely, suppose $S$ is in the convex hull of extreme rays of $H_n$.
Each of these rays can be realized in a weighted $K_N$ for some $N\le m(n)$.
For a specific conical combination of extreme rays giving $S$, 
consider the graph obtained by identifying all of their associated graphs at 
their terminal vertices, each with its edge weights multiplied by the coefficient 
in the associated conical combination. 
This graph clearly still realizes $S$,\footnote{This is 
basically the statement that any flow network problem with multiple source/sink
vertices can be equivalently reformulated in terms of a single supersource/supersink 
vertex connected to each of the sources/sinks with edges of infinite capacity. In 
our case, however, it is preferable to simply merge together terminals of the same 
type into a single ``superterminal'', rather than having unnecessary infinite-weight 
edges.}
and can be thought of as a $K_N$ 
for some large but finite $N$ with many zero-weight edges omitted. 
However, the bound on $N$ in Corollary \ref{cor:bound} guarantees that by 
merging vertices this graph can be reduced to one with  $N\le m(n)$, thus proving convexity.
\end{proof}

Of course, $H_{n}$ inherits the
$Sym_{n+1}$ symmetry discussed above. As a result, when considering the extremal
structure of $H_n$, we will only specify single representatives of symmetry
orbits under $Sym_{n+1}$.

Almost nothing is known about the complexity of computational problems related 
to $H_n$.
\begin{problem}
Given a vector 
$q\in \mathbb{Z}^{2^n-1}$, what is the complexity of 
deciding if $q \in H_n$?
What is the complexity of deciding
whether $qx \ge 0$ is satisfied for all $x \in H_n$? 
\end{problem}

\subsection{Representations of the HEC}\label{ssec:reps}

A basic result of polyhedral geometry is that
any polyhedral cone $C$ can be represented by a non-redundant set of 
facet-defining inequalities, which we can write as $Ax \ge 0$ for a suitably
dimensioned matrix $A$ and variables $x$, and is called an
\textit{$H$-representation}. One can also represent $C$ by a non-redundant list $E$
of its extreme rays, such that $C = \conv\{E\}$, which is called a 
\textit{$V$-representation}. Both representations are unique up to row scaling.
In this paper we are concerned with computing the $H$- and $V$-representations of
$H_n$. Normally, for cones (or polyhedra) arising in discrete optimization,
we have available one or the other of the representations. However,
this is not the case for $H_n$. To proceed we will initially try to find both
\textit{valid inequalities} for $C$, which are those satisfied by all rays in $C$,
and to find \textit{feasible rays} of $C$. A set of valid inequalities forms
an \textit{outer approximation} of $H_n$ and a set of valid rays forms
an \textit{inner approximation}. The following well known basic result
shows when such sets respectively constitute an $H$- and $V$-representation
(see, e.g., Schrijver \cite{S99}).
\begin{proposition}
\label{H-V}
For a given polyhedral cone $C$,
let $Ax \ge 0$ be a non-redundant set of valid inequalities and let
$E$ be a non-redundant set of feasible rays. Then:
\begin{enumerate}[label=(\alph*),leftmargin=2\parindent]
    \item \label{prop:V} An extreme ray of $Ax \ge 0$
is an extreme ray of $C$ if it is feasible for $C$.
    \item \label{prop:H} A facet of $\conv\{E\}$ is a facet of 
    $C$ if it is valid for $C$.
    \item \label{prop:HV} $E$ is precisely the set of extreme rays of 
    $Ax \ge 0$ if and only if they respectively constitute $V$- and 
    $H$-representations of $C$.
\end{enumerate}
\end{proposition}

Parts \ref{prop:V} and \ref{prop:H} lead to a kind of bootstrapping process which 
terminates once \ref{prop:HV} can be applied.
This is described in detail
in Section \ref{H5}, but to illustrate we now look at some small values of $n$.

For $n=1$, the $S$-vectors are the nonnegative 
real numbers. There is one facet $S_1 \ge 0$ and one extreme ray with $S_1 = 1$, and
Proposition \ref{H-V}\ref{prop:HV} is readily verified. Notice that this 
simple $n=1$ extreme ray can be represented in $K_2$ with the single edge $(1,2)$ having
weight $1$. More generally, any $K_N$ where two distinct singletons $i,j\in\extn$ share a
unit-weight edge and all other weights are zero will be called an $(i,j)$ 
\textit{Bell pair}. The $S$-vector of such a Bell pair has nonzero $S_I=1$ if and 
only if either $i\in I$ or $j\in I$, but not both.

For $n=2$ we have $S=(S_1, S_2, S_{12})$ and a valid inequality
$S_1 + S_2 \ge S_{12}$.
This follows since the union of cuts for two terminals $i,j\in\extn$ is always a cut for 
the union of the two terminals $\{i,j\}$.
Hence the latter's min-cut cannot have 
larger weight than the sum of the weights of the two other cuts. 
The full orbit of $S_1 + S_2 \ge S_{12}$ contains $S_1 + S_{12} \ge S_{2}$ and 
$S_{12} + S_2 \ge S_{1}$. All three are related by the $S_3$ symmetry of permutations
of $[2;N]$ and the identification of $S_{[2;N]\sminus I} = S_I$ for all 
$I\subseteq[2;N]$. In the context of information theory, the first one is known 
as \textit{subadditivity} (SA), while the latter two are called the Araki-Lieb inequalities.

Subadditivity generalizes to a valid inequality for any disjoint, non-empty subsets of 
terminals $I,J$ with $I \cup J\subset\extn$ to give 
\begin{equation}
\label{eq:SA}
    \text{SA:} \qquad S_I + S_J \ge S_{IJ}.
\end{equation}
We exclude $I \cup J = \extn$ since in this case SA reduces to nonnegativity.
Every inequality in the enlarged symmetry orbit $Sym_{n+1}$ of \eqref{eq:SA} clearly 
remains valid. For any disjoint subsets of terminals $I,J\subseteq\extn$, those of 
Araki-Lieb type take the form $S_I + S_{I\cup J} \ge S_{J}$. 
The qualitative difference between SA and Araki-Lieb is that in the former the disjoint 
subsets $I,J\subseteq\extn$ do not contain the sink, 
whereas in the latter one of them does, and 
whenever $N\in I \subseteq\extn$ one identifies $S_I = S_{\extn\sminus I}$. 
For future convenience, we introduce
\begin{equation}\label{eq:mi}
    \mi{I:J}S = S_I + S_J - S_{IJ},
\end{equation}
known as the \textit{mutual information} in the physics community. This way, SA 
corresponds to the nonnegativity of the mutual information $\mi{I:J}S\ge 0$.

Proceeding as suggested by 
Proposition \ref{H-V}\ref{prop:V}, we can compute the extreme rays arising 
from \eqref{eq:SA} by the action of $Sym_3$, together with the $3$ 
nonnegativity inequalities (in fact, the latter are redundant and can be ignored).
Doing so we obtain $3$ extreme rays related by 
symmetry, of which one is $S=(1,1,0)$. 
This can be represented in $K_3$ by a $(1,2)$ Bell pair. 
Obviously, this can also be obtained from the $(1,2)$ Bell pair
for $n=1$ in $K_2$ by adding a new vertex with all edge weights to it zero. 
This is a process called a \textit{zero-lifting} of extreme rays and which we discuss 
in detail in Section \ref{zerolift}.
So if we let $E$ be the set of $3$ 
output rays and $Ax \ge 0$ be the $3$ SA inequalities we are again 
done by Proposition \ref{H-V}\ref{prop:HV}.

\section{Extreme rays}
\label{rays}

Since we do not have an $H$-representation of $\HNn$ we have no direct
way to compute its extreme rays.
In this section we introduce a lifting of $\HNn$ to a cone for which
we can explicitly write an $H$-representation. Computing the extreme
rays of this cone and then making a projection allows us to compute
a superset of the extreme rays of $\HNn$. This construct provides the first 
systematic procedure for obtaining $H_n$, 
thus significantly improving on the random searches used so far in building the 
HEC \cite{Bao:2015bfa,Cuenca:2019uzx}. We also describe a
zero-lifting operation, which allows known extreme rays for $S$-vectors 
defined on $n$ terminals to generate extreme rays for those defined on
$n+1$ terminals.

\subsection{A lifting of the HEC}
\label{rlift}
For integers $n$ and $N$ such that $3 \le n+1 \le N$, consider the following system of inequalities
whose variables are the $S$-vector entries $S_I$ and the edge weights $w(e)$ of $K_N$:
\begin{subequations}
\label{SCnn}
\begin{alignat}{3}
\label{SC}
S_I &~\le~ \norm{C(W)}, \qquad &&\forall~W\subseteq [N-1],~\es\ne I\subseteq[n] 
\textit{ s.t. } I=W \cap [n], \\
\label{nn}
w(e) &~\ge~ 0, \qquad &&\forall~e\in E_N.
\end{alignat}
\end{subequations}
Note that for each cut $C(W)$ there is precisely one $I$ such that $I=W \cap [n]$.
Since we are excluding the case $I=\es$,
this implies that \eqref{SC} contains $2^{N-1}-2^{N-n-1}$ inequalities.
There are an additional $N(N-1)/2$ nonnegative inequalities from \eqref{nn}.
Since each $\norm{C(W)}$ is just a sum of weights $w(e)$ 
for every $e\in C(W)$, together with the $2^n-1$ variables $S_I$,
there are a total of $M_{N,n}=2^n-1+N(N-1)/2$ variables involved 
in \eqref{SCnn}. 
Let $(S,w)$ denote a vector of length $M_{N,n}$ 
representing these variables. 
\begin{definition}
The cone $\PNn$ denotes the set of all $(S,w)$ satisfying \eqref{SCnn}.
\end{definition}

Since $\PNn$ is given explicitly by \eqref{SCnn}, it is a rational cone.
Note that the variables $S_I$ in $\PNn$ are not bounded from below. One could 
add the inequalities $S_I \ge 0$ but this greatly increases the complexity of 
the cone, as explained below.

The cones $\HNn$ and $\PNn$ cannot be directly compared,
since they are defined in different dimensional spaces.
So we first define a lifting of $\HNn$ by
\begin{equation}
    H^+_{N,n} = \conv~\{(S,w) \in P_{N,n} \;:\; S \in H_{N,n} \}.
\end{equation}
which is by definition a subset of $\PNn$ and whose projection onto the $S$ 
coordinates is $\HNn$. We similarly define $H^+_n$.
An example in Appendix \ref{newS} shows that $H^+_{N,n}$ is in general non-convex
without the convex hull operator.
It is easy to see that $\PNn$ contains $M_{N,n}$ \textit{trivial} extreme rays. 
These are formed by setting either one $S_I =-1$ or one $w(e)=1$, and all other variables zero. 

Each extreme ray of $\HNn$ is the projection of a non-trivial extreme ray of $\PNn$, as we now show:
\begin{theorem}
\label{thm1}~
\begin{enumerate}[label=(\alph*),leftmargin=2\parindent]
    \item \label{thm:Sw} If $(S,w)$ is a non-trivial extreme ray of $\PNn$, 
    then $S \in \HNn$.
    \item \label{thm:HP}  If $S$ is an extreme ray of $\HNn$, then there is
    a weight vector $w$ such that $(S,w)$ is a non-trivial extreme 
    ray of $\PNn$.
\end{enumerate}
\end{theorem}

\begin{proof}~
\begin{enumerate}[label=(\alph*),leftmargin=2\parindent]
    \item \label{prf:Sw} 
        Suppose that $(S,w)$ defines a non-trivial extreme ray of $\PNn$. 
        Then there must be a set of at least $M_{N,n}-1$ inequalities 
        in \eqref{SCnn} satisfied as equations whose solutions have the form 
        $\lambda (S,w)$, with $\lambda \ge 0$.
        Each $S_I$ must be present in at least one of these equations or else it 
        could be increased independently of the others and the resulting vector 
        $(S',w)$ would still be a solution of the equations but not of that form.
        This in turn implies that $S$ satisfies \eqref{Sdef} for the given weight 
        function $w$. Hence $S \in \HNn$.
    \item \label{prf:HP}
        Suppose $S$ is an extreme ray of $\HNn$. Since $S$ satisfies \eqref{Sdef},
        all of its values are nonnegative and there must exist a corresponding 
        weight assignment $\bar{w}$. We now define a face $F$ of $\PNn$ by 
        intersecting it with the hyperplanes
    \begin{subequations}
    \begin{alignat}{3}
        \label{eq:sieq}
        S_I &~=~ \norm{C(W)}, \qquad &&\forall~W \text{ achieving the minimum 
        in \eqref{Sdef}}, \\
        \label{eq:we}
        w(e) &~=~ 0, \qquad &&\forall~e\in E_N \textit{ s.t. } \bar{w} (e)=0.
    \end{alignat}
    \end{subequations}
    Each $S_I$ appears in at least one equation. $F$ is defined by a set of 
    extreme rays of $\PNn$ but none of these can be a trivial ray of the type 
    $S_I=-1$ since $S_I = \norm{C(W)} \ge 0$. There may be a trivial extreme 
    ray of the type $w(e)=1$ as long as 
    $\bar{w}(e) \ne 0$ and the edge $e$ does not appear in any of the cuts 
    $C(W)$ in the system of hyperplanes. Suppose that there are $s$ of these 
    and write each of them as $1_e$. Also, denote 
    the non-trivial extreme 
    rays of $\PNn$ that lie on $F$ by $(S^i,w^i)$, with $i\in[t]$.
    From part \ref{prf:Sw} above, we have that 
    $S^i\in \HNn$ for all $i\in[t]$.
    Writing $(S,w)$ as a conical combination of the extreme rays that define $F$,
    \begin{equation}
        (S,w) =\sum_{i=1}^t \lambda_i (S^i,w^i)+\sum_{j=1}^s \mu_j 1_{e_j}, 
        \qquad \lambda_i, \mu_j \ge 0,
    \end{equation}
    we deduce that $S=\sum_{i=1}^t \lambda_i S^i$. Since $S$ is an extreme ray 
    of $\HNn$ it follows that $S=S^i$ for all $i$ for which $\lambda_i >0$. 
    For each such $i$, $(S,w^i)$ is a non-trivial extreme ray of $\PNn$.
\end{enumerate}
\vspace{-10pt}
\end{proof}

It was initially hoped that non-trivial extreme rays of $\PNn$ would project to extreme rays
of $\HNn$, but an example in Appendix \ref{newS} show that this is not the case. 
Also, a direct projection of $\PNn$ onto its $S$ coordinates does not 
give $\HNn$. For any given values of the $S$ coordinates, a feasible solution to 
inequalities \eqref{SCnn} can be obtained by choosing any suitably large 
$w$ coordinates.
Hence the projection is simply the whole space $R^{2^n-1}$.
Nevertheless, we do obtain a method in principle for obtaining a complete description
of $H_n$:
\begin{corollary}
\label{cor:V-des}
A $V$-description of $H_n$ can be obtained by computing the extreme rays
$(S,w)$ of $P_{m(n),n}$, projecting to the $S$ coordinates, and removing
both the trivial and the redundant rays.
\end{corollary} 
\begin{proof}
By Theorem \ref{thm1}\ref{thm:HP} every extreme ray $S\in H_{m(n),n}$ 
appears in some non-trivial extreme ray $(S,w)$ of $P_{m(n),n}$.
Projecting the latter to the $S$ coordinates produces all extreme rays
of $H_{m(n),n}$. Redundant rays can be removed by linear programming.
Since $H_n = H_{m(n),n}$, the result follows by Corollary \ref{cor:convex}.
\end{proof}

This important new result allows for a constructive derivation of $H_n$ 
(for which no explicit representation is known) starting from $P_{m(n),n}$ 
(for which a $V$-representation is straightforward to write down).
Since $P_{m(n),n}$ is a rational cone, as remarked earlier, it follows that $H_n$ is also.
This fact, proven in Proposition $7$ of \cite{Bao:2015bfa} by a different technique,
means extreme rays and facets always admit integral representations.
With our current upper bound on $m(n)$ the
computation is impractical except for small values of $n$.
It does not help here to include the extra inequalities $S_I \ge 0$
because this introduces a large number of new extreme rays $(S,w)$ of $P_{N,n}$
which are not valid pairs.
For example, while $P_{6,4}$ has $50$ extreme rays of which $15$ are trivial, 
adding nonnegativity yields $49915$ extreme rays, all but $35$ of which 
do not yield valid $(S,w)$ pairs. 
A computationally lighter method to test whether a given $S$-vector
is realizable is as follows:
\begin{theorem}
\label{direct}
If $\bar{S} \in \HNn$, then there is a weight vector
$\bar{w}$ such that $(\bar{S},\bar{w})$ is a vertex of the polyhedron $Q$ 
defined as the intersection of $\PNn$ and the hyperplanes
$S=\bar{S}$.\footnote{The converse of this is false,
as $Q$ generally has many vertices $(\bar{S},w)$ for which $w$ does not realize $\bar{S}$.}
\end{theorem}
\begin{proof}
First we suppose that $\bar{S} \in \HNn$. Then there exists a weight vector
$\bar{w}$ for $K_N$ that realizes $\bar{S}$. 
By construction $(\bar{S},\bar{w})$ is contained in $Q$ and so $Q$ is non-empty.
Since all components of $S$ in $Q$ are fixed, $Q$ is a possibly unbounded
polyhedron with vertices of the form $(\bar{S},w)$ and possibly additional 
extreme rays of the form $1_{e_j}$ for edges $e_j$ 
that do not appear in any 
minimum-weight 
cut $W$ obeying $W\cap[n]=I$ and $\norm{C(W)}=S_I$. 
We can write
\begin{equation}
\label{bb}
    (\bar{S},\bar{w}) = \sum_{i=1}^t \lambda_i \; (\bar{S},w^i),
    \qquad \sum_{i=1}^t \lambda_i=1, \qquad \lambda_i \ge 0,
\end{equation}
for a set of $t$ vertices $(\bar{S},w^i)$ of $Q$. For a cut $W$ in $K_N$, let 
$\norm{\bar{C}(W)}$ and $\norm{C^i(W)}$ denote its weight using $\bar{w}$ and 
$w^i$, respectively. Since $\bar{w}$ is a realization of $\bar{S}$, for each 
$\es \ne I \subseteq [n]$ there exists a min-cut $W_I$ such that 
$\bar{S}_I=\bar{C}(W_I)$.
Now, by \eqref{bb} and the linearity of the cut weight function,
\begin{equation}
    \norm{\bar{C}(W_I)}= \sum_{i=1}^t \lambda_i \; \norm{C^i(W_I)},
    \qquad \sum_{i=1}^t \lambda_i=1, \qquad \lambda_i \ge 0.
\end{equation}
Hence by \eqref{SC} we must also have 
$\bar{S}_I=\norm{C^i(W_I)}$ for all $i\in[t]$.
It follows that each $(\bar{S},w^i)$ is a vertex of $Q$ representing $\bar{S}$.
\end{proof}

We now show how these theorems can help in determining the extremal structure of $H_3$.
For example,
we can compute the extreme rays of $P_{5,3}$, project onto the $S$-coordinates,
delete the trivial rays and remove redundancy, getting a $V$-representation of $H_{5,3}$.
It contains
$17$ extreme rays in $7$ dimensions.
The $H$-representation of $H_{5,3}$ is easy to compute and contains
$7$ facets. One facet is new and the other $6$ are SA inequalities:
\begin{equation}
\label{SA3}
    S_i + S_j \ge S_{ij}, \qquad i \ne j \in [3;N],
\end{equation}
where $N$ is the sink and we recall that $S_{[3;N] \sminus I} = S_I$. 
Earlier we saw that for $n=2$ there is one SA orbit of $3$ facets: 
$S_1 + S_2 \ge S_{12}$, $S_1 + S_{12} \ge S_2$, and $S_{12} + S_2 \ge S_1$. 
Although the first one remains a facet of the form of \eqref{SA3} for $n=3$, 
the other two do not, which may seem surprising.
We return to this in Section \ref{flift} when discussing the lifting 
of facets (see Proposition \ref{prop:sal}).

The cone $H_{5,3}$ has one new facet which is inequality \eqref{eq:mmi} below. 
If we can prove this is valid for $H_3$ then, by Proposition \ref{H-V}\ref{prop:HV}, 
we will have complete $H$- and $V$-representations. How to prove validity of an 
inequality is discussed in Section \ref{facets}. 

\subsection{Zero-lifting extreme rays}
\label{zerolift}
Extreme rays for $\PNn$ remain extremal for larger values of $N$ 
as the following result describes:

\begin{proposition}
\label{zero}
If $(S,w)$ is an extreme ray of $\PNn$, then, by adding suitably many new 
weight coordinates set to zero, it is an extreme ray $(S,w')$ of $P_{N',n}$ for
every $N' > N$. Hence $\PNn$ is a projection of $P_{N',n}$.
\end{proposition}

\begin{proof}
Let $(S,w)$ define an extreme ray of $\PNn$ and consider the base graph $K_{N+1}$.
We extend the weight vector $w$ by adding $N$ new coordinates to get a
vector $w'$ of length $(N+1)N/2$.
We set $w'(i,j)=w(i,j)$ for $1 \le i < j < N$ and $w'(i,N+1)=w(i,N)$ for $1 \le i < N$.
All edges containing vertex $N$ receive weight zero in $w'$, and vertex $N+1$ is the 
new sink. Since $N$ is now a bulk vertex, it may participate in a cut $W$,
but since $\norm{C(W)} = \norm{C(W \cup \{N\})}$, it will never change its total weight.
Hence $(S,w') \in P_{N+1,n}$. 
Since $(S,w)$ defines an extreme ray of $\PNn$ we can
choose $M_{N,n}-1$
inequalities in \eqref{SCnn}
which, when 
satisfied as equations,
have solutions in $P_{N,n}$ of the form $\lambda (S,w)$ with $\lambda \ge 0$.
To these equations we add the $N$ equations $w(i,N)=0$.
The resulting system has solutions in $P_{N+1,n}$ of the form $\lambda (S,w')$ 
with $\lambda \ge 0$,
and the extremality of the ray defined by $(S,w')$ follows.
\end{proof}

Extreme rays can also be preserved under the addition of new terminal vertices as follows. 
Given $(S,w) \in \PNn$, we define its \textit{zero-lift} as the vector 
$(S',w') \in  P_{N+1,n+1}$,
where $S'$ has dimension $2^{n+1}-1$ and $w'$ has dimension $(N+1)N/2$, by
\begin{subequations}
\begin{alignat}{3}
\label{Sext}
S'_{ \{ n+1 \}}&=0, \qquad S'_{I \cup \{ n+1 \} } = S'_I =S_I,
\qquad \es \ne I \subseteq [n], \\
\label{wext}
w'(i,n+1) &= 0, \qquad i\in [n], \qquad w'(i,j)=w(i,j), \qquad 1 \le i < j \le N.
\end{alignat}
\end{subequations}
Similarly, $S'\in H_{N+1,n+1}$ above defines the \textit{zero-lift} of the given $S\in\HNn$.
A vector $x$ that satisfies an inequality $qx \ge 0$ as an equation is called
a \textit{root} of that inequality.

\begin{proposition}\label{raylift}
If $(S,w)$ is an extreme ray of $\PNn$, then its zero-lift $(S',w')$
is an extreme ray of $P_{N+1,n+1}$.
\end{proposition}
\begin{proof}
Let $(S,w)$ define an extreme ray of $\PNn$. For each $\es \ne I \subseteq [n]$ choose
an inequality from \eqref{SC} for which it is a root. These
are linearly independent inequalities since the $S_I$ coordinates define minus the identity matrix.
To these inequalities add all those
from \eqref{nn} for which $(S,w)$ is a root. Since $(S,w)$ is an extreme ray
we have a linearly independent set of $2^n + N(N-1)/2 -2$ such inequalities. 
Call this system $L$.
We now add terminal $n+1$ and define $S'$ and $w'$ as above. 
It is easy to verify that $(S',w') \in P_{N+1,n+1}$. Note $P_{N+1,n+1}$
has $2^n + N$ more dimensions than $\PNn$ and we will augment $L$ by this many
linearly independent inequalities. Firstly, for each $\es \ne I \subseteq [n]$ the 
inequality previously
chosen will also be satisfied as an equation when $I$ is replaced by $I \cup \{n+1\}$.
This gives an additional $2^n -1$ inequalities which are linearly independent from the
others in $L$ since the new variables $S_{I \cup \{n+1\}}$ again form minus the 
identity matrix.
To these we add $N$ equations $w(i,n+1)=0$ for $1 \le i \le n$ and the equation 
$S_{\{n+1\}} = \norm{C(\{n+1\})}$.
This gives $L$ the required number of tight linearly independent constraints.
By construction, their solution is $\lambda (S',w')$ with $\lambda \ge 0$, 
which proves that $(S',w')$ defines an extreme ray of $P_{N+1,n+1}$.
\end{proof}

The following result, though not unexpected, is proved here for the first time:
\begin{proposition}\label{hlift}
If $S$ is an extreme ray of $H_n$, then its zero-lift $S'$
is an extreme ray of $H_{n+1}$. Hence $H_n$ is a projection of $H_{n+1}$.
\end{proposition}

\begin{proof}
If $S$ is an extreme ray of $H_n$, then it is a root of a linearly independent set of
$2^n-2$ of its facet inequalities. 
Clearly $S'$ is also a root of these inequalities.
Additionally, $S'$ is a root of the following $2^n-1$ instances of SA,
\begin{equation}
\label{eqn:hlift}
    S_I' + S_{\{n+1\}}' \ge S_{I \cup \{n+1\}}',
\end{equation}
for all $\es\ne I\subseteq[n]$, and also obeys $S_{\{n+1\}}'=0$.
This gives $2^n$ more equations which are linearly independent
and are also independent of the former $2^n-2$ because they independently
involve the new, distinct variables $S_{I \cup \{n+1\}}'$ for every $I\subseteq[n]$.
Since $S'$ satisfies $2^{n+1}-2$ independent valid inequalities as equations,
it is an extreme ray of $H_{n+1}$ by Proposition \ref{H-V}\ref{prop:V}.
\end{proof}

To illustrate the use of the results in this subsection we consider the case $n=4$.
If we compute $P_{6,4}$ and remove the trivial rays 
we have $35$ extreme rays. When we project onto the $15$ coordinates $S_I$
and remove redundancy, there remain $20$ extreme rays. 
Only $5$ of these are new, whereas the other $15$ come 
from zero-lifts of the two extreme-ray classes that define $H_3$. 
The convex hull of this set
is bounded by $20$ facets that are examples
of what is called zero-lifting from the two facet classes for $H_3$.
We will define this process and prove that it preserves
validity and facets in Section \ref{flift}.
By this result and Proposition \ref{H-V}\ref{prop:HV}
we have obtained the $H$- and $V$-representations of $H_4$.

\section{Valid inequalities and facets}
\label{facets}

As remarked in Section \ref{defs}, there is no general explicit $H$-representation 
known for $H_n$, although it can in principle be computed by using the method of 
Corollary \ref{cor:V-des} and then converting the resulting 
$V$-representation into an 
$H$-representation. 
In Section \ref{ilp} we give an integer linear program (ILP) for testing whether an inequality
is valid over $H_n$ or not.
However, to prove validity would require solving an ILP whose size depends on $m(n)$ 
and so is impractical with current bounds.
The main result of this section is a tractable method known as 
\textit{proof by contraction} to prove inequalities valid for $H_n$.
By exhibiting the required number of extreme rays it is then possible
to prove they are facets.

A general inequality $q S \ge 0$ over $H_n$ is specified by a vector $q\in\RR^{{2^n-1}}$.
Let $K\subseteq[n]$ be the subset of terminals appearing in it, i.e. $i\in K$ if and only if 
$q_I \ne 0$ for some $I\ni i$.
If $\abs{K}<n$, one can turn it into an inequality over $H_{\abs{K}}$ by relabelling terminals 
$K\to \left[\,\abs{K}\right]$, if necessary. We say that an inequality over $H_n$ is in 
\textit{canonical form} if $\abs{K}=n$, and write it canonically as
\begin{equation}\label{eq:ineq}
    \sum_{l=1}^{L} \alpha_l S_{I_l} \ge \sum_{r=1}^{R} \beta_r S_{J_r},
\end{equation}
where $L$ and $R$ are respectively the number of positive and negative entries in the 
vector $q$, and for all $l\in[L]$ and $r\in[R]$, the coefficients $\alpha_l,\beta_r>0$ 
and the sets $I_l,J_r\subseteq[n]$ are distinct and non-empty. Because $H_n$ is a rational 
cone, the normalization of \eqref{eq:ineq} of any inequality of interest is always set such that all 
coefficients $\alpha_l,\beta_r>0$ are together coprime integers.

At the end of Section \ref{ssec:reps} we discussed the cases $n=1,2$. 
Recall that $H_1$ is $1$-dimensional and corresponds to a nonnegative
half-line. Its only facet is $S_1 \ge 0$, which trivially follows from nonnegativity 
of the weights in \eqref{HNndef}.
For $n=2$, we saw that the resulting $3$-dimensional cone $H_2$ is a simplex
bounded by the $3$ facets in the symmetry orbit of the SA inequality \eqref{eq:SA}.
We discussed $H_3$ at the end of Sections \ref{rlift} and \ref{zerolift}. Apart from 
the SA orbit, containing $6$ facet inequalities, an additional inequality was discovered 
to bound $H_3$ \cite{Hayden:2011ag}:
\begin{equation}\label{eq:mmi}
    S_{12} + S_{13} + S_{23} \ge S_{1} + S_{2} + S_{3} + S_{123}.
\end{equation}
This is known in physics as the \textit{monogamy of mutual information} (MMI) due to 
its rewriting using \eqref{eq:mi} as $\mi{1:23}S \ge \mi{1:2}S + \mi{1:3}S$. 
Since one can write submodularity as $\mi{1:23}S \ge \mi{1:2}S$, by nonnegativity 
of $\mi{1:3}S\ge 0$ one sees that MMI is a strictly stronger inequality.
The proof of the validity of \eqref{eq:mmi} will be presented in 
Section \ref{sec:contract} (see Table \ref{tab:mmi}) as 
an example of a general combinatorial proof method for valid inequalities of $H_n$.
This will show that $H_3$ has $7$ facets and so is also a simplex.
As we saw at the end of Section \ref{zerolift}, no new inequalities arise for $n=4$. 
However, for $n\ge3$ note that MMI acquires a more general form which we show
is valid for $H_n$: for disjoint non-empty 
subsets $I,J,K \subseteq [n]$,
\begin{equation}\label{eq:mmig}
    \text{MMI:} \qquad S_{IJ} + S_{IK} + S_{JK} \ge S_{I} + S_{J} + S_{K} + S_{IJK},
\end{equation}
which is valid for $H_n$, as we show in Section \ref{flift}.

\subsection{Proof by contraction}
\label{sec:contract}

We now describe the proof-by-contraction method for proving validity of inequalities for $H_n$. 
Although our description is complete,
we refer the reader to \cite{Bao:2015bfa} and \cite{Bao:2020zgx} for 
more details on its derivation. By making use of antichains, as in 
Corollary \ref{cor:bound}, we obtain a stronger result than previous ones.
Our discussion will be exemplified with MMI as given in \eqref{eq:mmi}.

To set the stage, let $Q_m=\{0,1\}^m\subset\RR^m$ denote the (vertices of) the unit 
$m$-cube and refer to $x\in Q_m$ as a \textit{bitstring}. At times, it will also be 
useful to think of $Q_m$ as an $m$-ary Boolean domain. Given some vector of positive 
entries $\gamma\in\RR_+^m$, we can turn $Q_m$ into a metric space with distance function
$d_\gamma$ by endowing it with a weighted Hamming norm $\norm{\,\cdot\,}_{\gamma}$ via
\begin{equation}\label{eq:df}
    d_\gamma(x,x') = \norm{x-x'}_\gamma, 
    \qquad \norm{x}_\gamma = \sum_{k=1}^m \gamma_k \abs{x_k}.
\end{equation}

Consider now a candidate inequality in canonical form over $H_n$, written as in \eqref{eq:ineq}.
Encode each side of it into $n+1$ \textit{occurrence vectors}, $x^{(i)}\in Q_L$ and 
$y^{(i)}\in Q_R$ for $i\in\extn$, with entries
\begin{equation}\label{eq:occur}
    x^{(i)}_l = \delta(i \in I_l), \qquad y^{(i)}_r = \delta(i \in J_r),
\end{equation}
where $\delta$ is a Boolean indicator function, i.e. it yields $1$ or $0$ depending 
on whether its argument is true or false, respectively. Clearly, the occurrence vectors 
for $i=N\notin[n]$ are all-$0$ vectors. For inequality \eqref{eq:mmi}, 
the $i\in[n]$ occurrence vectors are the following bitstrings:
\begin{equation}\label{eq:mmioc}
\begin{aligned}
    x^{(1)} &= (1,1,0), \qquad y^{(1)} = &(1,0,0,1), \\
    x^{(2)} &= (1,0,1), \qquad y^{(2)} = &(0,1,0,1), \\
    x^{(3)} &= (0,1,1), \qquad y^{(3)} = &(0,0,1,1).
\end{aligned}
\end{equation}
Bitstrings are a bookkeeping device for partitioning the vertex set $[N]$ of $K_N$ 
into specific disjoint subsets suitable for studying a given candidate inequality.
In particular, consider the minimal min-cuts $W_{I_l}$ for $l\in[L]$ associated to every term 
on the left-hand side of \eqref{eq:ineq}. Each of the $2^{L}$ different bitstrings 
$x\in Q_{L}$ indexes a disjoint vertex subset $W(x)\subseteq [N]$ defined by
\begin{equation}\label{eq:Wx}
    W(x) = \bigcap_{l=1}^L W_{I_l}^{x_{l}}, \qquad W^{b} =
    \begin{cases}
        W \qquad &\text{if \;$b=1$}, \\
        W^\comp \qquad &\text{if \;$b=0$}.
    \end{cases}
\end{equation}
The attentive reader will notice that these $W(x)$ sets would be precisely the atoms introduced 
in \eqref{Wpart} when proving Corollary \ref{cor:bound} if one were to iterate the intersection 
over all possible $2^n-1$ non-empty subsets of $[n]$. In the current discussion, one need only 
consider the pertinent $L$ subsets $I_l\subseteq[n]$ involved in the left-hand side 
of \eqref{eq:ineq}. The resulting $W(x)$ sets are again disjoint by construction, i.e. 
$W(x)\cap W(x') = \es$ unless $x=x'$. The converse is certainly not true though, as expected 
from Corollary \ref{cor:bound}. In particular, $W(x)$ will be empty whenever the family of sets 
$\mI(x) = \{I_l\subseteq [n] \;:\; x_l = 1\}$ is not a pairwise intersecting upper set in 
$\mI_L = \{I_l\subseteq [n] \;:\; l\in[L]\}$. The relation between this 
statement and the one in Corollary \ref{cor:bound} that refers to upper sets in $Bool_n$ is 
better understood in terms of their associated antichains. Namely, $W(x)$ will be empty whenever 
the minimal elements in $\mI(x)$ are not a pairwise intersecting antichain in $\mI_L$, which holds 
if and only if the same is true in $Bool_n$. In other words, there is a non-trivial $W(x)$ 
precisely for every pairwise intersecting antichain in $Bool_n$ that is also so in $\mI_L$, 
and thus the number of relevant bitstrings will be considerably smaller than $M(n)$.

The discussion above motivates introducing the subset $A_n(\mI_L) \subseteq Q_{2^n-1}$ of all 
bitstrings $x\in A_n(\mI_L)$ such that $\mI(x)$ is a pairwise intersecting antichain in $\mI_L$.
Crucially, these suffice to characterize the minimal min-cuts for $I_L\in\mI_L$ in any $K_N$, 
in the sense that these are all reconstructible via
\begin{equation}\label{eq:WIWx}
    W_{I_l} = \bigcup_{x:x_l = 1} W(x),
\end{equation}
where the union here, and in all that follows next, runs over all bitstrings $x\in A_n(\mI_L)$ 
subject to the given conditions. Furthermore, one can use these vertex sets $W(x)$ to construct 
(not necessarily minimum) cuts for any subset of terminals $J\subseteq \bigcup_{l=1}^L I_l \subseteq[n]$. 
To see this, let $i\in[n]$ be any one of the terminals involved in the subsets $I_l\in\mI_L$. We want to 
find which of the $W(x)$ sets the vertex $i$ lands on. Since $i\in W_{I_l}$ if and only if $i\in I_l$ by 
the definition of a cut for $I$, it follows that $i\in W(x)$ if and only if $x_l=1$ precisely when $I_l\ni i$
and $x_l=0$ otherwise. In other words, the bitstring we are after is precisely the occurrence vector 
$x^{(i)}\in Q_L$ defined in \eqref{eq:occur}, and we thus have $i\in W(x^{(i)})$.

\begin{lemma}\label{lem:cut}
    Let $f : Q_m \to \{0,1\}$ be an $m$-ary Boolean function. Given some collection of 
    terminal subsets $\mI_L\ni I_l$ and the partitioning of $[N]$ defined in \eqref{eq:Wx}, 
    construct the vertex set
    \begin{equation}
        U^f = \bigcup_{x:f(x)=1} W(x).
    \end{equation}
    Then, for any subset of terminals $J\subseteq \bigcup_{l=1}^L I_l \subseteq [n]$, we have
    \begin{equation}
        U^f\cap [n] = J \qquad\iff\qquad f(x^{(i)})=\delta(i\in J), \quad \forall~i\in[n].
    \end{equation}
\end{lemma}

\begin{proof}
    A trivial rephrasing of $U^f\cap[n] = J$ is that, for $i\in[n]$, one has $i\in J$ if and only if 
    $i \in U^f$. Since $i\in W(x^{(i)})$ for every $i\in[n]$ and all $W(x)$ are disjoint, it 
    follows that $i\in U^f$ if and only if $W(x^{(i)}) \subseteq U^f$. Finally, since by construction 
    $W(x^{(i)})\subseteq U^f$ if and only if $f(x^{(i)})=1$, the desired result is obtained.
\end{proof}

The min-cut edges $C(W_I)$ can also be conveniently organized in terms of bitstrings via
\begin{equation}\label{eq:edgexx}
    E(x,x') = \{ (i,j)\in E_N \;:\; i\in W(x) \text{ and } j\in W(x') \}.
\end{equation}
Because the $W(x)$ vertex sets are disjoint, so are the $E(x,x')$ edge sets for any distinct pair of 
bitstrings $x,x'\in Q_L$. This leads to the following useful result for $C(W_I)$:

\begin{lemma}\label{lem:edges}
    The edges of a min-cut $W_{I_l}$ for some $I_l \in \mI_L$ and their total weight are, respectively,
    \begin{equation}\label{eq:mcdecom}
        C(W_{I_l}) = \bigcup_{x, x' : x_l \ne x_l'} E(x, x'), \qquad
        \norm{C(W_{I_l})} = \sum_{x, x'} \abs{x_l - x_l'} \abs{E(x, x')},
    \end{equation}
    where the index sets are unordered pairs of bitstrings $x,x'\in A_n(\mathcal{I}_L)$.
\end{lemma}

\begin{proof}
    By definition, an edge $(i,j)\in C(W_{I_l})$ if and only if $i\in W_{I_l}$ and $j\in W_{I_l}^\comp$. 
    Using \eqref{eq:WIWx}, one can write $W_{I_l} = \bigcup_{x:x_l=1} W(x)$ and, similarly, 
    $W_{I_l}^\comp = \bigcup_{x:x_l=0} W(x)$. Hence $(i,j)\in E(x,x')$ is contained 
    in $C(W_{I_l})$ if and only if $x$ and $x'$ differ in their $l^{\text{th}}$ bit $x_l\ne x_l'$. 
    It follows that $C(W_{I_l})$ can be constructed by joining all edge sets $E(x,x')$ with 
    bitstrings $x, x'\in A_n(\mathcal{I}_L)$ such that $x_l \ne x_l'$, thereby proving the first 
    equation in \eqref{eq:mcdecom}. 
    Furthermore, since all $E(x,x')$ are disjoint for distinct pairs of bitstrings, the total weight of 
    their union reduces to the sum over the total weights of every $E(x,x')$ involved, which is 
    precisely what the second equation computes.
\end{proof}

Given two metric spaces $(M,d)$ and $(M',d')$, we call $f : M \to M'$ a 
\textit{$d$-$d'$ contraction map} if 
\begin{equation}\label{eq:cprop}
    d'(f(x),f(y)) ~\le~ d(x,y), \qquad \forall~x,y\in M.
\end{equation}
The general proof method can now be stated:
\begin{theorem}\label{thm:contraction}
    Inequality \eqref{eq:ineq} is valid for $H_n$ if there exists a $d_\alpha$-$d_\beta$ contraction map
    \begin{equation}\label{eq:contra}
        f: A_n(\mI_L) \to Q_R,
    \end{equation}
    satisfying $f(x^{(i)}) = y^{(i)}$ for all $i\in\extn$.
\end{theorem}

\begin{proof}
	Associate a cut $U_{J_r}$ to each subsystem $J_r$ that appears on the right-hand side 
	of \eqref{eq:ineq} by using the map $f$ to pick which sets $W(x)$ to include in the 
	definition of $U_{J_r}$ as follows:
	\begin{equation}
	U_{J_r} = \bigcup_{x:f(x)_r = 1} W(x).
	\end{equation}
	That this indeed obeys the cut condition $U_{J_r} \cap [n] = J_r$ is guaranteed by Lemma
	\ref{lem:cut} and the fact that $f$ is required to respect occurrence vectors, i.e.
	$f(x^{(i)}) = y^{(i)}$ for every $i\in\extn$.
	\begin{equation}
    	\sum_{l=1}^{L} \alpha_l {S}_{I_l}
    	= \sum_{l=1}^{L} \alpha_l \, \norm{C(W_{I_l})}
    	= \sum_{x, x'} \abs{E(x, x')} \, \sum_{l=1}^L \alpha_l \abs{x_{l} - x'_{l}}
    	= \sum_{x, x'} \abs{E(x, x')} \, d_{\alpha}(x,x').
	\end{equation}
	Similarly, for the $U_{J_r}$ cuts one has
	\begin{equation}
	\sum_{r=1}^{R} \beta_r \norm{C(U_{J_r})} = \sum_{x,x'} \abs{E(x, x')} \, d_\beta(f(x),f(x')).
	\end{equation}
	Therefore, by hypothesis, the contraction property of $f$ implies
	\begin{equation}\label{eq:cont}
	    \sum_{l=1}^L \alpha_l \norm{C(W_{I_l})} ~\ge~ \sum_{r=1}^R \beta_r \norm{C(U_{J_r})}.
	\end{equation}
	Because every set $U_{J_r}$ is a cut for each $J_r$ appearing on the right-hand 
	side of \eqref{eq:ineq}, by minimality $\norm{C(U(J_r)} \ge S(J_r)$ for every 
	$r\in[R]$. Hence the right-hand side of \eqref{eq:cont} is no smaller than that of \eqref{eq:ineq}.
	Finally, since their respective left-hand sides are equal, validity of \eqref{eq:ineq} follows.
\end{proof}

This theorem was proved in \cite{Bao:2015bfa} (Theorem $8$) in a somewhat weaker form.
Whereas the domain of their contraction map is 
$Q_L$, enumerating all subsets of $\mI_L\subseteq Bool_n$, 
we reduce this to $A_n(\mI_L)\subseteq Q_L$, enumerating only 
the pairwise intersecting antichains in 
$\mI_L$ (cf. Corollary \ref{cor:bound}). This leads to a reduction
of the worst-case complexity of the search space.

As an example, a contraction map which proves validity of \eqref{eq:mmi} is shown in Table 
\ref{tab:mmi}.\footnote{In fact, this contraction map which proves \eqref{eq:mmi} is unique. 
This is generically not the case for larger-$n$ facets, for which there usually exists 
many contraction maps compatible with the requirements of Theorem \ref{thm:contraction}.}
One easily checks that occurrence 
vectors are respected, e.g. for $3\in[n]$ we have $(0,1,1)\mapsto (0,0,1,1)$, which matches 
\eqref{eq:mmioc}. Iterating through every pair of rows, one can also check that the contraction 
property holds. Here the vectors defining the distance function are $\alpha = (1,1,1)$ and 
$\beta=(1,1,1,1)$  for left and right, respectively. For instance, occurrence 
vectors $1$ and $3$ in the domain give $d_\alpha(x^{(1)},x^{(3)}) = 1 + 0 + 1 = 2$, while 
their images
$d_\beta(f(x^{(1)}),f(x^{(3)})) = 1 + 0 + 1 + 0 = 2 \le d_\alpha(x^{(1)},x^{(3)})$. 
Notice that the map $f$ need not be injective nor surjective.

\begin{table}[ht]
\setlength{\tabcolsep}{.14cm}\renewcommand{\arraystretch}{.97}
	\centering
	\begin{tabular}{c || c | c | c||c | c | c | c}
		& $S_{12}$ & $S_{13}$ & $S_{23}$ & $S_{1}$ & $S_{2}$ & $S_{3}$ & $S_{123}$ \\ \hline
		$0$ & 0 & 0 & 0 & 0 & 0 & 0 & 0 \\
		 ~  & 0 & 0 & 1 & 0 & 0 & 0 & 1 \\
		 ~  & 0 & 1 & 0 & 0 & 0 & 0 & 1 \\
		$3$ & 0 & 1 & 1 & 0 & 0 & 1 & 1 \\
		 ~  & 1 & 0 & 0 & 0 & 0 & 0 & 1 \\
		$2$ & 1 & 0 & 1 & 0 & 1 & 0 & 1 \\
		$1$ & 1 & 1 & 0 & 1 & 0 & 0 & 1 \\
		 ~  & 1 & 1 & 1 & 0 & 0 & 0 & 1 \\
	\end{tabular}
	\caption{
	Representation of the contraction map which proves the MMI inequality \eqref{eq:mmi}. 
	The left-most column labels the occurrence 
    vectors shown in \eqref{eq:mmioc}, including the one for $N\sim0$. The top row labels 
    bitstring entries, separating domain (left) from codomain (right). For the domain, $S_{I_l}$ 
    labels entries $x_l$, $l\in[L]$ for $x\in Q_L$ and, 
    for the codomain, $S_{J_r}$ 
    labels entries $y_r$, $r\in[R]$ for $y\in Q_R$. Every row represents one entry of the 
    map $f : x\mapsto y$ by listing all entries as $\{x,y\}$.}
	\label{tab:mmi}
\end{table}

The proof of Theorem \ref{thm:contraction} is constructive and, as shown in \cite{Bao:2015bfa},
leads to an algorithm for finding a contraction map or showing none exists. The
enumeration of all contraction maps is prohibitively expensive in all but very small cases.
However, the authors developed a greedy technique for partial search
which is successful in finding a map, when one exists. Indeed, we have found this method
very powerful in proving new inequalities valid for $H_6$.
For proving an inequality is invalid,
the previously-mentioned ILP approach, which will be presented in Section \ref{ilp},
is also very effective.

Theorem \ref{thm:contraction} provides a robust sufficient condition for an inequality to be valid, 
but it is not a necessary one. For example, even after exhausting
all of the possibilities given by the theorem, it was not able to prove the
validity of this inequality over $H_5$:
\begin{equation}
\begin{aligned}
\label{eq:beast}
    3 S_{123}&+3 S_{124}+S_{125}+S_{134}+3 S_{135}+S_{145}+S_{234}+S_{235}+S_{245}+S_{345} \ge \\
    2 S_{12}&+2 S_{13}+S_{14}+S_{15}+S_{23}+2 S_{24}+2 S_{35}+S_{45}
    +2 S_{1234}+2 S_{1235}+S_{1245}+S_{1345}.
\end{aligned}
\end{equation}
However, this inequality can be proved valid by expanding the codomain of $f$ by replacing 
coefficients greater than one on the right-hand side by a sum of terms with unit coefficients. 
For instance, a term like $2 S_{I}$ gets replaced by $S_{I}+S_{I}$, with the obvious
generalization applied to larger coefficients. Theorem \ref{thm:contraction} still applies 
and this time the desired contraction map does exist, thereby proving validity of \eqref{eq:beast}
By expanding the right-hand side there are more possible images for the contraction map, while 
the number of contraction conditions remains fixed. This may explain why this approach worked well
here and in other cases we have tried for larger $n$.

This mild generalization of the proof technique of Theorem \ref{thm:contraction} has been 
remarkably successful in proving inequalities for $n=6$, which motivates the following problem:

\begin{problem}
\label{prob:cont}
    In \eqref{eq:ineq}, if we replace terms $\beta_r S_{J_r}$ with $\beta_r \ge 2$
    by $\sum_{i=1}^{\beta_r} S_{J_r}$ and accordingly adjust $R$ to $\sum_{r=1}^{R}\beta_r$, 
    does Theorem \ref{thm:contraction} provide a necessary condition for validity over $H_n$?
\end{problem}

\subsection{Zero-lifting of valid inequalities and facets}
\label{flift}

Given an inequality $qS\ge 0$ over $H_n$, let $K\subseteq[n]$ be the subset of terminals
appearing in it. Then consider a family of disjoint, non-empty subsets $\{I_i\subseteq[n+1]\}_{i\in K}$ 
(not necessarily spanning).
The \textit{zero-lifting} of the inequality given by this family is obtained by replacing
each singleton $i\in I$ in every $S_I$ in $qS\ge 0$ by its corresponding $I_i$.
For example, the zero-lifting of $S_1 + S_2 \ge S_{12}$
from $H_3$ to $H_4$ corresponding to $I_1=\{2,3\}$ and $I_2=\{1,4\}$ yields $S_{23}+S_{14} \ge S_{1234}$.
The zero-lift where $I_i = \{i\}$ for every $i\in K$ is called the \textit{trivial zero-lift}.

\begin{proposition}\label{ineqlift}
    If an inequality $qS \ge 0$ is valid for $H_n$, then any
    zero-lift $q' S' \ge 0$ is valid for $H_{n+1}$.
\end{proposition}

\begin{proof}
Assume $qS \ge 0$ is valid for $H_n$. Proceed by contradiction by supposing 
$q$ has a zero-lift $q'$ such that $q' S' < 0$ for some $S'\in H_{n+1}$.
Such $S'$ must be realized by some weight map $w$ applied to $K_N$ for some
$N$. In this $K_N$, contract each terminal set $I_i$ to the vertex in $I_i$ 
with the minimum label,
combining parallel edges and summing their weights into a single edge, and deleting any loops.
Let $S$ be the realized $S$-vector in the new graph. 
We have $q'S' = q S \ge 0$, the desired contradiction.
\end{proof}

Before discussing lifting facets we need to recall some terminology from
Proposition \ref{prop:full}, in particular the weighted star graphs and
the construction of matrix $A^{n+1}$ in \eqref{eq:A3}. We will make frequent 
use of the square matrix $D^n$ of size $2^n-1$, defined by
\begin{equation}
    D^n_{I,J} = \abs{I\cap J}, \qquad\es\ne I\subseteq[n],\quad\es\ne J\subseteq[n+1],~n+1 \in J,
\end{equation}
and use the notation $D^n_J$ to refer to row $J$ of $D^n$.
An inequality $qS \ge 0$ in $\RR^{2^n-1}$ is called \textit{balanced} if $D_{\{j\}}q=0$ for all 
$j \in [n]$. By definition, balance is invariant 
under any permutation of terminals $[n]$, but need not be so under permutations of the extended 
terminals $\extn$. For instance, $S_1+S_2 \ge S_{12}$ is balanced but $S_1 + S_{12} \ge S_2$ is not.
Balance is equivalent to a seemingly stronger condition:

\begin{lemma}\label{lem:starbal}
    An inequality $qS \ge 0$ in  $\RR^{2^n-1}$ is balanced if and only if $D^nq=0$.
\end{lemma}
\begin{proof}
    Obviously, $D^nq=0$ implies balance. For the converse, writing out row $J$ of $D^nq$,
    \begin{equation}
        D^n_Jq = \sum_{\es\ne I\subseteq[n]} q_I \abs{I\cap J} = \sum_{j=1}^{n} \delta_J^{j}
        \sum_{\es\ne I\subseteq[n]} q_I \delta_I^{j} = \sum_{j=1}^{n} \delta_J^{j} \; D_{\{j\}}^n q,
    \end{equation}
    where $\delta_I^i=\delta(i\in I)$ 
    (cf. \eqref{eq:occur}) and we used $\abs{I\cap J} = \sum_{k=1}^n \delta_I^k \delta_J^k$.
    So $D^nq=0$ by balance.
\end{proof}

What follows is a new result which relates the trivial zero-lifting of facets to the notion of balance:
\begin{proposition}\label{prop:facetlift}
If $qS \ge 0$ is a balanced facet of $H_n$, then its trivial zero-lift $q'S' \ge 0$
is a balanced facet of $H_{n+1}$.
\end{proposition}
\begin{proof}
Suppose $qS \ge 0$ is a balanced facet of $H_n$.
Then it is a valid inequality of $H_{n+1}$ by Proposition \ref{ineqlift}.
We adopt the notation of Proposition \ref{prop:full} and build a matrix $A^{n+1}$
with the structure in \eqref{eq:matrices}, except it will now have only $2^{n+1}-2$ rows.
Let $B^n$ consist of $2^n-2$ linearly independent roots of $qS\ge 0$ as rows,
so that the first $2^n-2$ rows of $A^{n+1}$ become precisely their zero-lifts.
The trivial zero-lift has $q_I'=0$ for every $I\ni n+1$, so these are all roots
of $q'S'\ge0$ as well.
Since $qS\ge0$ is balanced we have $D^n q=0$ by Lemma \ref{lem:starbal}. 
So the corresponding rows of $A^{n+1}$ are roots of $q'S' \ge 0$ too.
The final row is also and so $A^{n+1}$ contains $2^{n+1}-2$ roots of $q'S'\ge 0$.
Performing the same column operations as in Proposition \ref{prop:full}, the resulting block
matrix (cf. $\tilde{A}^{n+1}$) shows that $A^{n+1}$ has maximal rank $2^{n+1}-2$.
Since $D_{\{j\}}^{n+1}q'=D_{\{j\}}^{n}q=0$, the lifted facet is also balanced.
\end{proof}

The following proposition clarifies the situation for subadditive inequalities,
which include the non-balanced Araki-Lieb inequalities in their orbits:

\begin{proposition}
\label{prop:sal}
For all $n \ge 2$, a zero-lift of a subadditive inequality \eqref{eq:SA} gives a facet if and 
only if, using the symmetry $S_{\extn \sminus I} = S_I$, it can be put in the singleton SA form
\begin{equation}
\label{eq:sal}
    S_i + S_j \ge S_{ij}, \qquad i \ne j \in \extn.
\end{equation}
\end{proposition}
\begin{proof}
Since singleton SA is a balanced facet of $H_2$, so is $S_i + S_j \ge S_{ij}$ for $H_n$ by
Proposition \ref{prop:facetlift}, as it can be obtained by iterating trivial zero-lifts and
making a $Sym_n$ permutation at the end.

For the converse, if a subadditive inequality is not in the form \eqref{eq:sal}, we may write it
as $S_I + S_{JK} \ge S_{IJK}$, for non-empty subsets $I$, $J$ and $K$. This inequality is the 
sum of three valid inequalities for $H_n$: the general SA inequality \eqref{eq:SA}, the general 
MMI inequality \eqref{eq:mmig} and $S_I + S_K \ge S_{IK}$. Therefore, it is not a facet of $H_n$.
\end{proof}

Apart from nonnegativity and the Araki-Lieb inequality associated to \eqref{eq:sal},
all known facets of $H_n$ are balanced. While balance is sufficient for 
trivial zero-lifts to preserve facets, a stronger condition is needed for general zero-lifts. 
A balanced inequality $qS \ge 0$ is \textit{superbalanced} if every inequality in its symmetry
orbit under $Sym_{n+1}$ permutations of $\extn$ is balanced \cite{Hubeny:2018ijt,He:2020xuo}. 
Since balance is
invariant under permutations of $[n]$, it is in fact only necessary to 
check if exchanges of every $i\in[n]$ with $N$ yield balanced inequalities.
Orbits of superbalanced inequalities are referred to as superbalanced. 
For example, SA in \eqref{eq:SA} for $I,J\subseteq[n]$ is balanced but not superbalanced 
and MMI in \eqref{eq:mmi} is superbalanced. 
According to results stated in \cite{He:2020xuo}, besides the singleton SA orbit, every orbit of 
facets of $H_n$ for $n\ge2$ is superbalanced.

We can now generalize Proposition \ref{prop:facetlift} to arbitrary zero-lifts.
For an inequality $qS\ge 0$ in $\RR^{2^n-1}$, let
\begin{equation}\label{eq:qtilde}
    \tilde{q}_I = \sum_{I \subseteq J \subseteq [n]} q_J.
\end{equation}

\begin{lemma}\label{lem:sb}
    An inequality $qS\ge 0$ in $\RR^{2^n-1}$ 
    is superbalanced if and only if $\tilde{q}_I=0$ for every $I\subseteq[n]$ with $\abs{I}\le 2$.
\end{lemma}
\begin{proof}
    That $\tilde{q}_{\{i\}}=0$ for all $i\in[n]$ is just the definition of balance, which is an invariant 
    property under permutations of $[n]$. Permutations of $\extn$ also allow for reflections
    $j\leftrightarrow N$ for each $j\in[n]$. Using $S_{\extn\sminus K} = S_{K}$, the $S$-vector entries 
    $S_I'$ after reflection are related to the $S_I$ before reflection by $S_J'=S_J$ and
    $S_{J\cup\{j\}}'=S_{[n]\sminus J}$ for $J\not\ni j$.
    For example, $1\leftrightarrow N$ for $n=3$
    gives $S'=(S_{123},S_2,S_3,S_{13},S_{12},S_{23},S_{1})$.
    The coefficients $q_I$ in $qS\ge 0$ behave accordingly. Under a $j\leftrightarrow N$ reflection,
    \eqref{eq:qtilde} gives $\tilde{q}_{\{j\}}'=\tilde{q}_{\{j\}}$, while for $i\ne j$ one gets
    \begin{equation}
        \tilde{q}_{\{i\}}'
        = \sum_{i \in J \subseteq [n]\sminus\{j\}} (q_{J\cup\{j\}}' + q_J')
        = \sum_{i \in J \subseteq [n]\sminus\{j\}} (q_{[n]\sminus J} + q_J)
        = \sum_{j \in J \subseteq [n]\sminus\{i\}} q_{J} + \sum_{i \in J \subseteq [n]\sminus\{j\}} q_J.
    \end{equation}
    The first sum is over all $q_J$ such that $J\ni j$ but $J\not\ni i$, so it differs 
    from $\tilde{q}_{j}$ precisely by $\tilde{q}_{\{i,j\}}$.
    Similarly for the second sum exchanging $i\leftrightarrow j$, so
    \begin{equation}\label{eq:sbsim}
        \tilde{q}_{\{i\}}' = \tilde{q}_{\{i\}} + \tilde{q}_{\{j\}} - 2 \tilde{q}_{\{i,j\}}.
    \end{equation}
    After the exchange $j\leftrightarrow N$, $q'S\ge 0$ is balanced if and only if $\tilde{q}_{\{i\}}'=0$ 
    for all $i\in[n]$. Therefore $qS\ge 0$ is superbalanced if and only if 
    $\tilde{q}_{\{i\}}=\tilde{q}_{\{i\}}'=0$ for all $i\in[n]$. Applied to \eqref{eq:sbsim}, 
    this means $qS\ge 0$ is superbalanced if and only if $\tilde{q}_{\{i\}}=\tilde{q}_{\{i,j\}}=0$ 
    for all $i,j\in[n]$.
\end{proof}

We now show that any zero-lift of a superbalanced facet can actually be built solely out of 
trivial zero-lifts combined with permutations of the extended terminals, both of which preserve 
facets. Superbalance is needed for such permutations to preserve balance and thus keep 
Proposition \ref{prop:facetlift} applicable. 
It is also important in what follows that, as is clear from Lemma \ref{lem:sb} 
and the form of \eqref{eq:qtilde}, balance and superbalance are properties which are shared by 
inequalities related by trivial zero-lifts.
We first observe that the trivial 
zero-lift from $H_n$ to $H_{n+1}$ can be thought of as treating the
new terminal $n+1$ as a duplication of the sink 
(since the sink does not appear anywhere in $qS\ge 0$, neither does $n+1$ in $q'S'\ge 0$). 
But by the symmetry of $H_n$ under $Sym_{n+1}$ permutations of $\extn$, we could analogously 
consider letting $n+1$ duplicate any other terminal. Let us call such a generalization of a trivial 
zero-lift where any one extended terminal becomes a doubleton and the rest remain singletons 
a \textit{simple zero-lift}. We obtain the following novel result:

\begin{theorem}\label{thm:facet}
    If $qS \ge 0$ is a superbalanced facet of $H_n$, then any zero-lift
    $q'S'\ge 0$ is a superbalanced facet of $H_{n+1}$.
\end{theorem}
\begin{proof}
If $qS\ge 0$ is a facet inequality over $H_n$ involving a subset of terminals $K$ 
    with $\abs{K}<n$, then put it in canonical form as an inequality over $H_{\abs{K}}$.
Iterating Proposition \ref{hlift}, note that $H_{\abs{K}}$ is a projection of $H_n$.
Since $q$ is in canonical form, its coefficients are not changed in projecting it to  $H_{\abs{K}}$.
A standard result of polyhedral theory is that facets project to facets, so
$qS \ge 0$ is a superbalanced facet of $H_{\abs{K}}$.

Starting from $qS\ge 0$, one can get to $q'S'\ge 0$ as follows. If the original zero-lift had 
$\{i\}\mapsto I_i$, perform $\abs{I_i}-1$ simple zero-lifts appending terminals to $i$, and repeat 
for every $i\in[n]$. If less than $n+1-\abs{J}$ steps were required, reach all the way to 
$H_{n+1}$ via trivial zero-lifts. At that point, a suitable permutation of $[n+1]$ yields $q'S'\ge 0$.

We now show that every simple zero-lift used above can in fact be built solely out of permutations 
and trivial zero-lifts. In particular, the simple zero-lift involving $I_i=\{i,j\}$ is equivalently 
accomplished by exchanging $i\leftrightarrow N$, performing a trivial zero-lift, and then exchanging 
the new sink back with $i$. If the original inequality is superbalanced, the trivial zero-lift in 
this process is applied to a balanced inequality. Using Proposition \ref{prop:facetlift}, one ends 
up with a facet if one started with a facet. Furthermore, the latter is superbalanced if the former is.
Hence one can go from $qS\geq0$ to $q'S'\geq0$ via superbalance- and
facet-preserving steps.
\end{proof}

\section{Integer programs for testing realizability and validity}
\label{ilp}

This section describes novel methods for checking if an $S$-vector is realizable (and if so 
finding a graph realization), and for checking if a given inequality is valid.
A direct test of the realizability of an $S$-vector in $K_N$ can be performed by a feasibility test 
of a mixed integer linear program (ILP). Similarly, an inequality $qS \ge 0$ can be tested to see if 
it is valid for all $S$-vectors that can be realized in $K_N$. As noted earlier, the polyhedral
approach described so far does not force the minimum in \eqref{Sdef} to be realized by one of the 
inequalities \eqref{SC}. However, using binary variables this can be achieved and the feasibility of
the resulting system tested using ILP solvers such as 
\cplex, \glpsol~or \gurobi.

For any $N > n \ge 3$, we build a set of constraints, $\ILP_{N,n}$, whose feasible solution is the
set of all suitably-normalized, valid $(S,w)$ pairs on $n$ terminals realizable in $K_N$.
Firstly, note that for each $\es \ne I \subseteq [n]$, the number of cuts $W$ in $K_N$ that contain 
$I$ is $2^{N-\abs{I}-1}$. 
For each such $W$ and $I$, we introduce a binary variable $y_{W,I}$. Specifically, we consider the 
following system:

\tool{$\bf \ILP_{N,n}$}
\vspace{-.25cm}\hrule\vspace{.25cm}
\hspace{.5cm} For all $ \es \ne I \subseteq [n]$ and cuts $W\subseteq[N-1]$ in $K_N$ such that $ I =W \cap [n]$,
\vspace{-.2cm}
\begin{alignat}{2}
\label{ilpform}
S_I &~\le~ \norm{C(W)}, \\
\label{slack}
\norm{C(W)} &~\le~ S_I +\abs{W}~(N-\abs{W})~y_{W,I}, \\
\label{ysum} 
\sum_{W \cap [n]=I} y_{W,I} &~=~ 2^{N-\abs{I}-1} - 1, \\
\label{binary}
y_{W,I} &~\in~ \{0,1\}, \\
\label{bounds}
0 ~\le~ & w(e) ~\le~ 1, \qquad \forall~ e \in E_N.
\end{alignat}
\vspace{-.4cm}\hrule

\begin{proposition}
\label{prop:ilp}
A pair $(S,w)$ is valid in $K_N$ with all edge weights at most one if and only if there exists 
assignments to variables $y$ so that $\{S,w,y\}$ is a feasible solution to $\ILP_{N,n}$.
\end{proposition}

\begin{proof}
Suppose that $(S,w)$ is a valid pair in $K_N$ with all edge weights at most one. We will show that $y$
variables can be chosen so that $\{S,w,y\}$ constitutes a feasible solution to  $\ILP_{N,n}$.
Firstly, by assumption $w$ satisfies \eqref{bounds}. Next, since $(S,w)$ is a realization in $K_N$,
the upper bounds in \eqref{ilpform} are valid. For each $ \es \ne I \subseteq [n]$, choose one 
$W_I \subseteq [N-1]$ so that $W_I$ realizes a minimum in \eqref{Sdef}, and set $y_{W_I,I}=0$.
All other $y$ variables for this $I$ are set to $1$, thus satisfying \eqref{ysum} and \eqref{binary}.
Since  $y_{W_I,I}=0$, the corresponding equation \eqref{slack} gets zero as the 
second term in its right-hand side, and thus combines with \eqref{ilpform} into the required equation.
The remaining inequalities to verify are those in \eqref{slack} when $y_{W,I}=1$. Their validity follows 
from the fact that the cut $W$ in $K_N$ contains $\abs{W} (N-\abs{W})$ edges, each of weight at most one.

Conversely, let $\{S,w,y\}$ be a feasible solution of $\ILP_{N,n}$. For each $ \es \ne I \subseteq [n]$, 
\eqref{ysum} implies that there is a single variable, which we label $y_{W_I,I}$, having value zero.
The other $y$ values for this $I$ are one. Together with \eqref{ilpform}, this implies that
$S_I = \norm{C(W_I)}$ and that $S_I$ satisfies \eqref{Sdef}. So $(S,w)$ is a valid pair realized in 
$K_N$ with all edge weights at most one.
\end{proof}

We make use of this ILP formulation in two ways. Firstly, it can be used to test whether or not an 
$S$-vector is realizable in $K_N$ for a given $N$. To do this, we pre-assign the values from the given 
$S$-vector to the corresponding $S$ variables in $\ILP_{N,n}$, rescaled to values smaller than $1$. 
We may then run an ILP solver to test whether there is a feasible solution. 
If so, the values or the variables $w$ will give a realization
in $K_N$. Otherwise, one concludes that the given $S$-vector cannot be represented in any $K_{N'}$ with
$N'\le N$. Secondly, we may use the ILP to test whether an inequality $qS \ge 0$ is invalid for some 
$S$-vector realized in $K_N$ for a given $N$. This can be done by minimizing $z=qS$ over $\ILP_{N,n}$ 
and seeing if the optimum solution is negative. The computation can be terminated when the first feasible 
solution with $z<0$ is found, at which point $qS \ge 0$ is proven invalid.
We could prove that an inequality $qS \ge 0$ is
valid over $H_n$ by testing it with $\ILP_{m(n),n}$, but this ILP would be very
large with current bounds on $m(n)$.

In its first formulation, the ILP allows one to find the minimum value $N_{min}$ of $N$ for which an
$S$-vector is realizable in $K_N$. Given an $S$-vector, we call any such $K_{N_{min}}$
a \textit{minimum realization}. At fixed $n$, we define $m_{ext}(n)$ as the smallest integer such that 
all extreme rays of $H_n$, and hence of $H_{m_{ext}(n),n}$, are realizable in $K_{m_{ext}(n)}$ 
(cf. the definition of $m(n)$). For $1\le n \le 5$, the ILP shows that $m_{ext}(n)$ is
\begin{equation}
    2,~~3,~~5,~~6,~~11.
\end{equation}
Combining all extreme-ray graphs into a larger one by identifying them all at $\extn$ (cf. conically 
combining $S$-vectors), one can also see that for $1\le n \le 3$, $m(n)$ takes values $2$, $3$ and $5$. 
Namely, no bulk vertices are needed for $n=1,2$, and just a single one comes into play for $n=3$. 

The case $n=4$ is less trivial. There are two star-graph orbits of $5$ extreme rays each,
see Figure \ref{fig:eg} in Appendix \ref{app:rays}. These are $10$ extreme rays realizable in $K_{6}$,
which contains a single bulk vertex.
The other extreme rays of $H_4$ involve no bulk vertices. Hence, a convex combination of $15$ extreme rays 
may require a total of $10$ bulk vertices at most, which with the terminals and sink gives $m(4)\le 15$. 
This can be further improved as follows. The Bell-pair extreme rays span a subspace of dimension $10$,
and the star-graph extreme rays are confined to its $5$-dimensional orthogonal complement. 
Thus at most $5$ star-graph extreme rays are needed to conically span any interior ray of $H_4$, improving the
bound down to $m(4)\le 10$. It turns out that the simplicity of the specific extreme-ray graphs for $n=4$ 
in fact allows us to obtain the definite value $m(4)=6$. The reason for this is that the $S$-vector of 
any conical combination of these particular extreme-ray star graphs of $H_4$ can itself also be realized
on a star graph. This follows from the observation that all $n=4$ extreme-ray star graphs have identical 
minimal min-cuts: for every $\es\ne I\subseteq[n]$, they all have $W_I = I$ for $\abs{I}=1,2$ and 
$W_I = I\cup\{n\}$ for $\abs{I}=3,4$. Pictorially, this allows one to stack them all on top of each other, 
adding up their edge weights, so as to realize any combination of these star graphs by a star graph.

This discussion illustrates some strategies for obtaining tighter upper bounds on 
$m(n)$ based on knowledge of extreme rays or the value of $m_{ext}(n)$. Recall that the number of bulk 
vertices in $K_N$ is $N-n-1$. Regardless of how many extreme rays $H_n$ has, any interior 
ray may be a conical combination of at most $2^{n}-1$ of them. Since we can 
realize all extreme rays in $K_{m_{ext}(n)}$, we have
\begin{equation}\label{eq:mnb}
    m(n) \le (m_{ext}(n)-n-1)\times (2^{n}-1) + n + 1.
\end{equation}
For instance, since we know $m_{ext}(5)=11$, this gives $m(5) \le 161$, which is considerably better than 
the bound $m(5) < M(5)=2546$ given earlier.

We can do even better for $n=5$ by using explicit results about the dimensionality of the span of specific
extreme-ray orbits. The Bell pairs take care of $15$ dimensions which are not reached by any other extreme
ray without introducing any bulk vertices. There is a single orbit that requires $N=11$, and it consists of
$75$ extreme rays spanning a subspace of dimension $10$ of the remaining $16$ of $H_5\subset\RR^{31}$. The
largest-$N$ orbit spanning the other $6$ dimensions has $N=8$ and $360$ extreme rays. Hence the worst-case
scenario would require $10$ graphs with $N=11$ and other $6$ with $N=8$. The total number of vertices
carried by a combination of such graphs thus gives the bound $m(5) \le 74$. This is better than the more 
general one attained by \eqref{eq:mnb}, but requires complete knowledge of all extreme-ray graphs, not just 
of the number $m_{ext}(n)$.

\begin{problem}
    Find tighter bounds on $m_{ext}(n)$. In particular, does $\log_2 m_{ext}(n)$ admit an upper bound that 
    is polynomial in $n$?
\end{problem}

\section{Computing \texorpdfstring{$H$}{H}- and 
\texorpdfstring{$V$}{V}-representations of \texorpdfstring{$H_n$}{Hn}}
\label{H5}

To date, there existed no direct or algorithmic procedures for constructing $H_n$, 
and all results obtained for up to $n=5$ relied on random/heuristic searches.
This section provides two novel systematic methods for computing complete descriptions of $H_n$.
We illustrate them for $n=5$ and describe $H_5$ in detail in Appendix \ref{good}.\footnote{A partial description of $H_5$ was first obtained by \cite{Bao:2015bfa} and only four years later 
completed by \cite{Cuenca:2019uzx}; the approaches proposed here only take a few hours and additionally obtain provably minimum graph realizations of all extreme rays unknown to date.}
We also show how partial results for $n=6$ can be obtained by our methods. However, a complete description of $H_6$ appears to be beyond current computational capabilities.
Further upgrading our methods to obtain $H_6$ is the subject of work in progress with Bogdan Stoica, to be reported elsewhere.
The first method is a general formalization of the strategy used earlier for $n=3,4$, 
whereas the second one constructs $H_{n}$ starting from knowledge of $H_{n-1}$.

\subsection{Method 1}

To initialize this method we first set $k=2$.

\tool{\textbf{Method $\bf 1$}}
\label{mm1}

\begin{enumerate}[label=(\alph*),leftmargin=2.5\parindent,rightmargin=2\parindent]
\setlength\itemsep{-1pt}
\hrule\vspace{-.1cm}
\item\label{1b}
Generate the $H$-representation \texttt{P\{n+k\}-n.ine} of $P_{n+k,n}$ using \eqref{SCnn}. 
Convert
this to a $V$-representation \texttt{P\{n+k\}-n.ext}. 
\item\label{1c}
Delete the $2^n-1$ trivial extreme rays (see Theorem \ref{thm1}) and extract the $2^n-1$ 
coordinates corresponding to the variables of the $S$-vectors. 
Remove redundant rays to obtain
the $V$-representation 
\texttt{H\{n+k\}-n.ext} of $H_{n+k,n}$. This is an inner approximation of $H_n$.
\item\label{1d}
Compute the $H$-representation \texttt{H\{n+k\}-n.ine} of $H_{n+k,n}$ from \texttt{H\{n+k\}-n.ext}.
Using
the ILP method with of Section \ref{ilp} with $N \ge n+1$, 
reject facet orbits that are invalid for $K_N$, continuing until either a facet
is rejected or $N$ is too large for the ILP to solve. 
\item\label{1e}
Test any remaining facet orbits for which the validity is unknown
using the proof-by-contraction method of Section \ref{sec:contract}.
Generate the full orbits of the facets proved valid, getting
a cone \texttt{HV\{n+k\}-n.ine} which is an outer approximation of $H_n$.
\item\label{1f}
Compute the extreme rays \texttt{HV\{n+k\}-n.ext} of \texttt{HV\{n+k\}-n.ine}.
The orbits of $S$-vectors that appeared in
\texttt{P\{n+k\}-n.ext} give extreme rays of $H_n$ by Theorem \ref{thm1}\ref{thm:HP}.
The remaining orbits can be checked by the ILP method of Section \ref{ilp} with $N \ge n+1$ until
finding a realization or $N$ being too large for the ILP to solve.
If all extreme-ray orbits can be realized, then \texttt{HV\{n+k\}-n.ine} is an $H$-representation of $H_n$ 
and \texttt{HV\{n+k\}-n.ext} is its $V$-representation by Proposition \ref{H-V}\ref{prop:HV}.  
Otherwise, increment $k$ and return to step \ref{1b}. 
\vspace{.1cm}\hrule\vspace{.1cm}
\end{enumerate}

Applying Method 1 with $n=5$, one finds that $P_{7,5}$ has $83$ facets and $194$ extreme rays in $52$ 
dimensions. The resulting $H_{7,5}$ has $142$ extreme rays in $31$ dimensions, and its $H$-representation 
consists of $8952$ facets in $30$ orbits. All but $8$ of them are easily eliminated in 
step \ref{1d} and then proved valid in step \ref{1e}. These orbits give $372$ facets which 
define \texttt{HV7-5.ine}. Correspondingly, \texttt{HV7-5.ext} has $2267$ extreme rays falling into $19$ orbits.
All of the extreme rays are realizable for $N \le 11$, so the procedure terminates after a single iteration.
Note that we obtain a minimum realization of each extreme ray either in 
step \ref{1b} or \ref{1f}, wherever it appears first.

The vertex/facet enumeration problems in steps \ref{1b}, \ref{1d} and \ref{1f} utilized the 
code \normaliz\footnote{\url{https://www.normaliz.uni-osnabrueck.de}} $v.3.4.1$
on \mait\footnote{\mait: $2\times$ Xeon E5-2690 (10-core 3.0GHz), 20 cores, 128GB memory.}.
Steps \ref{1b} and \ref{1d} took only a few seconds, and step \ref{1f} took $23$ minutes.
Step \ref{1e} was run on a laptop\footnote{Dell XPS 15 7590, i7-9750H CPU @ 2.60GHz, 
6 cores, 12 threads, 32GB memory.} using a Mathematica $v.12.1$ 
implementation\footnote{Available upon request.} of the proof-by-contraction method.
Most runs were very fast, taking less than $4$ seconds, and all finished in no more than $16$ minutes.
The ILP runs in step \ref{1f} were performed with 
\cplex\footnote{\url{https://www.ibm.com/analytics/cplex-optimizer}} $v.12.6.3$, 
also on \mait, and normally completed in under $1$ minute, the longest run taking $18$ minutes. 
The filtration by symmetry generally takes just a few seconds.

Applying Method 1 with $n=6$ we run into computational issues as the vertex/facet enumeration
problems quickly become too big to solve with current software and hardware.
Nevertheless, we are still able to get useful results using only partial computations.
Starting with $k=2$,
$P_{8,6}$ has $154$ facets and $194$ extreme rays in $91$
dimensions. The resulting $V$-representation of
$H_{8,6}$ has $4361$ extreme rays in $63$ dimensions falling into
$21$ orbits. All $21$ orbits can be shown to define orbits of extreme rays of $H_6$ by testing their rank against lifts of $n=5$ inequalities to $n=6$.
Obtaining the $H$-representation of 
$H_{8,6}$ is computationally intractable with presently available
algorithms and hardware.
To get more extreme ray orbits we set
$k=3$ and constructed the $H$-representation of
$P_{9,6}$, which has 288 facets in 99 dimensions. It was not possible to
do a complete computation of it V-representation. 
The code \normaliz~ ran out of memory after about a 
day of computation, as did other double-description based methods. However we were able to
get partial results using the parallel reverse search based method \mplrs~contained
in \lrslib
\footnote{\url{http://cgm.cs.mcgill.ca/~avis/C/lrs.html}} $v.7.2$
which gives output in a stream. After about 6 months of computation with \mplrs~
using between $100$ and $200$
processors we obtained $213,225$ extreme rays 
which fall into $1066$ orbits. By construction, all of these extreme rays
are realizable in $K_9$. Together, these orbits generate about 3 million
extreme rays, however $460$ orbits become redundant when the full orbits
are considered. The $606$ non-redundant
orbits generate about 1.5 million extreme rays and $402$ of the orbits are
provably orbits of extreme rays of $H_6$ using lifted inequalities from $n=5$. 
Unfortunately it is not possible to compute the facets of such a large cone
with current methods.
It is important to note that \mplrs~ supports checkpoint/restart and continuing the computation
will continue the output stream until a complete $V$-representation is obtained.
Being based on reverse search,
computer memory is not a constraining factor.

\subsection{Method 2}

The second method is more sophisticated and involves working
with both outer and inner approximations of $H_n$, refining them until they are equal.
The outer approximation is initialized by choosing any set of valid inequalities for $H_n$,
not necessarily facets,
whose intersection is full dimensional.
The inner approximation is initialized by choosing any feasible set of
rays, not necessarily extreme, whose convex hull is also full dimensional. A strong way to initialize 
the outer approximation is to
zero-lift the superbalanced facets of $H_{n-1}$ in all possible ways, add to them singleton
SA, and generate their full orbits under $Sym_{n+1}$. By Theorem \ref{thm:facet} and
Proposition \ref{prop:sal}, these are all facets of $H_n$ and define a cone \texttt{OH1-n.ine}.
For a strong inner approximation, we
zero-lift the extreme rays of $H_{n-1}$ and generate their full orbits under $Sym_{n+1}$, 
which are all extremal in $H_n$ by Proposition \ref{hlift}.
Since this is not always full-dimensional, we
add the full orbits of the $S$-vectors from Proposition \ref{prop:full} not already included, 
and remove redundancies.
In general, this may only add the single orbit of size $n+1$ generated by $S^{[n]}$.
The resulting cone \texttt{IH1-n.ext} is an inner approximation of $H_n$. Set 
the iteration counter $k=1$.

\vspace{.25cm}
\tool{\textbf{Method $\bf 2$}}
\label{mm2}
\vspace{-30pt}
\begin{table}[H]
\setlength{\tabcolsep}{.015\textwidth}\renewcommand{\arraystretch}{1.15}
    \normalsize
    \begin{tabular}{ m{.013\textwidth} m{.42\textwidth} m{.42\textwidth} c}
        &{\center Outer\HRule[5pt]} &  {\center Inner\HRule[5pt]}& \\
        &Compute the $V$-representation \texttt{OHk-n.ext} of \texttt{OHk-n.ine}. 
        Check one extreme ray from each orbit to see if it is realizable by the ILP method of Section \ref{ilp}.
        If all rays are realizable, then \textbf{exit}.
        The full orbits of the realizable extreme rays define \texttt{OHVk-n.ext}. &
        Compute the $H$-representation \texttt{IHk-n.ine} of \texttt{IHk-n.ext}. Apply to it
        steps \ref{1d} and \ref{1e} of Method \hyperref[mm1]{1}, retaining inequalities proved
        by the contraction method of Section \ref{sec:contract}.
        If all inequalities are valid, then \textbf{exit}.
        The full orbits of the valid facets define \texttt{IHVk-n.ine}.&\\
        &$\hfill\downarrow\hfill$ & $\hfill\downarrow\hfill$ &\\
        &Merge \texttt{IHVk-n.ine} (and any other known valid inequalities) 
        with \texttt{OHk-n.ine} and remove redundancies to get \texttt{OH\{k+1\}-n.ine}. & 
        Merge \texttt{OHVk-n.ext} (and any other known realizable rays)
        with \texttt{IHk-n.ext} and remove
        redundancies to get \texttt{IH\{k+1\}-n.ext}. &\\
        &$\hfill\downarrow\hfill$ & $\hfill\downarrow\hfill$ &\\
        \multicolumn{4}{>{\centering\arraybackslash}p{.97\textwidth}}{Increment $k$ and return to 
        the first step of each respective subroutine.\HRule[6pt]}
    \end{tabular}
\end{table}
\vspace{-.3cm}

Note that the inner and outer procedures can be run
in parallel. After they both finish the first step, the newly computed data are exchanged, 
improving both the outer an inner approximations. If \textbf{exit} occurs, the corresponding
\texttt{ine} and \texttt{ext} descriptions give $H$- and $V$-representations of $H_n$, respectively.
In each subroutine, the second step allows for the incorporation of valid inequalities and/or rays 
obtained by other means, such as Method \hyperref[mm1]{1}.

Applying Method 2 with $n=5$, the starting cones \texttt{OH1-5.ine} and \texttt{IH1-5.ext} respectively consist of $80$ 
facets in $3$ orbits (that of singleton SA and $2$ of MMI, cf. Appendix \ref{app:facets}), 
and $66$ extreme rays in $4$ orbits (cf. Figure \ref{fig:eg} in Appendix \ref{app:rays}, and the $J=[n]$ star orbit).

We start with $k=1$ and describe steps in parallel.
In the outer run, \texttt{OH1-5.ext} has $3205$ extreme rays in $29$ orbits, out of which $16$ can 
be shown to be realizable with $N\le11$. Their orbits yield $1457$ feasible rays defining \texttt{OHV1-5.ext}.
In the inner run, \texttt{IH1-5.ine} has $157153$ facets in $346$ orbits, 
out of which one can show $8$ are valid and easily reject the rest.
Their orbits yield $372$ valid inequalities defining \texttt{IHV1-5.ine}.
In the second step it turns out that the outputs of the first step dominate in both cases.
So after the merges,
\texttt{OH2-5.ine} equals \texttt{IHV1-5.ine} and
\texttt{IH2-5.ext} equals \texttt{OHV1-5.ext}.

Setting $k=2$, in the outer run \texttt{OH2-5.ext} has $2267$ extreme rays in $19$ orbits, 
all of which are realizable with $N\le11$.
Hence \textbf{exit} is triggered, and the algorithm terminates returning \texttt{OH2-5} as the result for $H_5$.
If we continue the inner run we find that \texttt{IH2-5.ine} has $1182$ 
facets in $11$ orbits, out of which one can show $8$ 
are valid and easily reject the rest. These are the same $8$ orbits as before and so the algorithm exits
in the first outer step with $k=3$.

Conversions between cone representations again require vertex/facet enumeration. Those in the first 
iteration are immediate. Using \normaliz~on \mait, the computations of \texttt{IH2-5.ine} and \texttt{OH3-5.ext}
took about $75$ seconds and $25$ minutes, respectively. The cost of other computations was similar 
to that of their counterparts in Method \hyperref[mm1]{1}.

Applying Method 2 with $n=6$, the starting cones \texttt{OH1-6.ine} and \texttt{IH1-6.ext}
are in $63$ dimensions and respectively consist of $6503$
facets in $11$ orbits
and $15617$ extreme rays in $20$ orbits.
We start with $k=1$. Again we have to be satisfied with partial computation using \mplrs~
as the problem is too large for current computational methods to terminate in reasonable time.
In the outer run after about 10 days of computation we obtained
$3445$ extreme rays from \texttt{OH1-6.ext} belonging to 3288 distinct orbits. 
None of these 3288 orbits can be ruled out using lifted inequalities (by construction, because they come from \texttt{OH1-6.ine}), which means a priori we have literally 3288 candidates. 
Using a set of heuristically generated inequalities that we were able to prove valid via the contraction method, we reduce this list down to 55 candidates. Of these, $5$ orbits are easily seen to correspond to lifts of $n=5$ extreme rays.
Using \cplex~and the ILP method we found minimum realizations of all $55$ orbits:
$1$ in $K_7$ (lift of a Bell pair), $6$ in $K_{8}$ ($4$ are lifts of star graphs), $10$ 
in $K_{9}$, $13$ in $K_{10}$, $16$ in $K_{11}$ and
$9$ in $K_{12}$. 
All of these extreme rays are thus extreme rays of $H_6$ by construction.
Note that only those realizable in $K_n$ for $n \le 9$ could have appeared in the Method 1 run described.
As noted for Method 1, the computation of \texttt{OH1-6.ext} can be continued with additional
new orbits being produced as a stream until the computation is completed.
In the inner run, the computation of the $H$-representation of \texttt{IH1-6.ext}
was too big to
produce any useful output in two weeks of computation with \mplrs~using $160$ processors.

\subsection{Comparison of Method 1 and Method 2}
Although the inner steps of Method \hyperref[mm2]{2} may appear similar to Method \hyperref[mm1]{1},
they are in fact quite distinct. In the latter, the starting cone \texttt{H\{n+2\}-n.ext} only 
contains $S$-vectors realizable in $K_{n+2}$. Many of these will be non-extremal in $H_n$ and therefore absent 
from \texttt{IH1-n.ext}. Among those which are extremal, some may not be obtainable by zero-lift and thus 
not included in \texttt{IH1-n.ext} either. On the other hand, \texttt{IH1-n.ext} contains all extreme rays of 
$H_n$ coming from zero-lifts. These will generally include plenty which are not realizable in $K_{n+2}$ 
and hence not be contained in \texttt{H\{n+2\}-n.ext}. For example, for $n=6$, \texttt{IH1-6.ext} includes 
zero-lifts of extreme rays in the $5$ orbits of $H_5$ which are realizable 
in $K_N$ with $N \ge 9$ (see Table \ref{tab:rays} in Appendix \ref{app:rays}), 
none of which can possibly be in \texttt{H8-6.ext}.

Both methods may run into fundamental and/or practical issues.
For $n=5$, one is fortunate that the contraction method successfully proves valid 
the $8$ facet orbits of $H_5$. However, it remains a logical possibility that for larger $n$
this proof method is not a necessary condition for validity of some facets of $H_n$
(cf. Problem \ref{prob:cont} at the end of Section \ref{sec:contract}).
Specifically in Method \hyperref[mm1]{1}, it so happens that all
rays in \texttt{HV7-5.ext} are realizable using the ILP of Section \ref{ilp}.
For larger $n$, in practice it could be that even if all rays at step \ref{1f}
were realizable, the value of $N$ required could be too high for the ILP to be solved. 
Without good bounds on $m_{ext}(n)$, this possibility cannot be easily eliminated. 
Alternatively, it could be that some rays are indeed not realizable, meaning that the facet 
description in \texttt{HV\{n+k\}-n.ine} is incomplete. This would require incrementing $k$ 
and at least one further iteration.
As for Method \hyperref[mm2]{2}, we unfortunately have no proof of convergence using the strong
starting inputs suggested without the option to generate and add additional
valid inequalities and/or feasible rays in the second step.
There are various heuristic methods available to generate such additional inputs.
Another complication that affects these methods is the need to solve large convex hull/facet 
enumeration problems. All of these issues arise in one form or another in both methods 
already in the study of $H_6$.

The successful termination of either method relies on the finding of an $H$/$V$-representation 
of an inner/outer approximation of $H_n$ containing all of its facets/extreme rays. For instance, observe 
that in Method \hyperref[mm1]{1} all facets of $H_5$ were already discovered in step \ref{1b} and computed
explicitly in step \ref{1d} (along with other non-valid inequalities) from $H_{N,5}$ for just $N=7$. 
Similarly, Method \hyperref[mm2]{2} converged more easily through an inner approximation
\texttt{IH1-5} whose $H$-representation also contained all facets of $H_5$.
That $H_n$ is easier to obtain from an $H$-representation of an inner approximation is no accident.
This is because smaller $N$ for $K_N$ is needed to 
span all facets than to realize all extreme rays of $H_n$. This motivates the definition of $m_{ine}(n)$ 
as the smallest integer such that the $H$-representation of $H_{m_{ine}(n),n}$ contains all facets of $H_n$.

It is easily seen that $m_{ine}(n)=m_{ext}(n)$ for $1\le n\le 4$ and that $m_{ine}(n)\le m_{ext}(n)$ for 
larger $n$. For $n=5$, the cone $H_{6,5}$ turns out to miss some facets of $H_5$, but $H_{7,5}$ does contain
them all as we have seen in Method \hyperref[mm1]{1}. This shows that $m_{ine}(5)=7$, contrasting with the 
extreme rays, which have $m_{ext}(5)=11$. More generally, when Method \hyperref[mm1]{1} terminates, we have
$m_{ine}(n)=n+k$ and a minimum realization of each extreme ray, from which one also obtains $m_{ext}(n)$.
This makes the importance of $m_{ine}(n)$ manifest and motivates the following problem:
\begin{problem}
\label{prob:boundmine}
    Find tighter bounds on $m_{ine}(n)$. In particular, does $\log_2 m_{ine}(n)$ admit an upper bound that 
    is polynomial in $n$?
\end{problem}

\section{Conclusion}
Many of the important questions about the HEC remain open.
As stated formally throughout the paper in Problems \ref{prob:boundm} through \ref{prob:boundmine}, these include obtaining
an explicit description of either the $H$- or $V$-representation of $H_n$, and finding
the complexity of testing feasibility of rays and validity of inequalities. The current
bounds on the size of the complete graph that can realize all extreme rays of $H_n$
seem far from being tight, at least according to the limited experimental results
that we have. Similarly, our findings suggest that much smaller graphs may be sufficient to span
all facets of $H_n$, which strongly motivates understanding better the relative
complexity of the $H$- and $V$-representations of the HEC.
Ultimately, one would hope to obtain a more fundamental understanding of the HEC, such as in the form of the structural conjectures put forward recently in \cite{Czech:2021rxe,Fadel:2021urx,Hernandez-Cuenca:2022pst,Czech:2022fzb}.
In this work, we have laid the foundations for further exploration of these key questions and provided some useful tools for testing and proving such ideas.
Additionally, we have provided sharp computational tools
which allowed us to completely describe $H_5$ after just a few hours of computation
and produce significant new results for $H_6$.

\section*{Acknowledgments}
We thank
Patrick Hayden,
Temple He,
Veronika Hubeny,
Max Rota,
Bogdan Stoica
and
Michael Walter
for useful discussions.
We would also like to thank an anonymous referee for many comments and suggestions for
improving the paper.

\bibliographystyle{spbasic}
\bibliography{mp}
\appendix

\section{Origins and importance of the HEC in physics}
\label{physics}
The tools of convex geometry have long been applied to systematically study entropy
inequalities, from those obeyed by the Shannon entropy of random variables in
classical probability distributions,
to the ones that the von Neumann entropy of marginals of density
matrices of quantum states satisfy \cite{pippenger2003inequalities}.
As a measure of quantum entanglement, the study of the latter has
proven to be of paramount importance to the development of the field of quantum
information theory and, more generally, to the understanding of correlations in
quantum physics \cite{nielsen2002quantum}.

Although the finding of universal inequalities obeyed by general quantum states
has been elusive, significant progress has been made by the restriction of the
domain of the entropy function to specific subclasses of quantum states of
special relevance for which additional tools are at hand.
In the context of quantum gravity and holography,
one very important such class of
quantum states are those which admit a semi-classical description in terms of a
theory of gravity on a higher-dimensional spacetime.
More specifically, in such cases, the Anti-de Sitter/Conformal Field Theory (AdS/CFT) 
correspondence asserts that a holographic state of the CFT, defined on a \textit{boundary} 
spacetime $M$, has a gravitational \textit{bulk} dual on a spacetime $\mathcal{M}$ with $M$ 
as its conformal boundary, $\partial\mathcal{M}=M$. In the bulk, quantum entanglement of the 
CFT state acquires a  geometric character which has been understood to play a fundamental 
role in the very emergence of spacetime itself \cite{VanRaamsdonk:2010pw}. These findings 
rely on the much celebrated Ryu-Takayanagi (RT) prescription \cite{Ryu:2006bv,Rangamani:2016dms}, 
according to which the von Neumann entropy $S(R)$ of a spatial boundary region 
$R\subset M$ is given holographically by
\begin{equation}
\label{eq:RT}
    S(R) = \min_{\mathcal{R}\subset \mathcal{M}} \; \frac{\area(\mathcal{R})}{4 G \hbar},
\end{equation}
where $G$ is Newton constant, $\hbar$ is Planck constant, and
the minimization is over bulk hypersurfaces $\mathcal{R}$ in a time slice
homologous to $R$ relative to $\partial R$, i.e., subject 
to the condition $\partial \mathcal{R} = \partial R$.
This geometric character that the von Neumann entropy acquires in the bulk turns out to place 
strong constraints on the allowed entanglement structures of holographic states. In a remarkable paper, 
Bao et al. \cite{Bao:2015bfa} initiated a systematic exploration of these constraints with the 
objective of formalizing a set of conditions on quantum states to posses holographic duals.
These were formulated as entropy inequalities satisfied by the RT formula,
defining the facets of a polyhedral cone which was coined as the HEC.

More precisely, the HEC is a family of polyhedral 
cones $H_n$ labelled by an integer $n\ge 1$, all related by projections from 
larger to smaller $n$. 
Their work laid the ground for the finding of new results about the HEC
\cite{Cuenca:2019uzx,Czech:2019lps,He:2020xuo}, and also lead to further
generalizations and explorations of their methods
\cite{Hubeny:2018trv,Hubeny:2018ijt,Hernandez-Cuenca:2019jpv,Bao:2020zgx,Walter:2020zvt,Bao:2020mqq}.
Most of these developments relied on two remarkable results of \cite{Bao:2015bfa}: 
a proof of equivalence between holographic entropies obtained by the RT formula
and minimum cuts on weighted
graphs\footnote{Intuitively, the graph provides a discrete
tessellation of the manifold which encodes sufficient information about its
metric in the form of edge weights, with minimal surfaces and their areas
becoming minimum cuts and their weights, respectively -- see \cite{Bao:2015bfa}.},
and the invention of a combinatorial method to prove the validity of holographic entropy 
inequalities that we will review in Section \ref{facets}.
Crucially, their graph models allow for a complete study of the HEC from a purely combinatorial viewpoint
without reference to the geometric RT formula or quantum physics.

\section{Extremal structure of \texorpdfstring{$H_n$}{Hn} for 
\texorpdfstring{$1 \le n \le 5$}{1<=n<=5}}
\label{good}

Here we summarize the extremal structure of $H_n$ for all $1 \le n \le 5$ by showing
representatives of every orbit of both facets and extreme rays. Representatives of each orbit are 
picked as their lexicographical minimum.\footnote{The only exception to this is inequality
\ref{eq:cyclic} in Table \ref{tab:h5}, which is chosen for symmetry reasons.}
For extreme rays, we also present their minimum realizations,
exhibiting graphs where only edges of nonzero weight are shown.
At every $n$, we only include elements which are genuinely new and not coming from zero-liftings.
This is because these should always be included -- by Proposition \ref{hlift} the zero-lift of rays 
preserves all extreme rays, while by Theorem \ref{thm:facet} the zero-lift of inequalities preserves 
all superbalanced facets. As for SA, Proposition \ref{prop:sal} guarantees that precisely only instances 
involving just singletons in $\extn$ give rise to facets. It will thus be convenient to present results in
increasing order of $n$.

\subsection{Facets}
\label{app:facets}

At $n=1$ one just has one single-element orbit of a nonnegativity facet,
\begin{equation}
    S_1 \ge 0.
\end{equation}
For $n=2$, the cone is a simplex with $3$ facets in a single orbit of SA,
\begin{equation}
    S_1 + S_2 \ge S_{12}.
\end{equation}
Lifting to $n=3$, one gets $6$ facets in the orbit of the trivial zero-lift of SA. The cone 
becomes again a simplex due to the appearance of the new, totally symmetric facet of MMI
\begin{equation}\label{eq:mmia}
    S_{12} + S_{13} + S_{23} \ge S_{1} + S_{2} + S_{3} + S_{123}.
\end{equation}
There are no genuinely new inequalities for $n=4$. The trivial zero-lift of the SA facet gives 
a length-$10$ orbit. Every zero-lift of \eqref{eq:mmia} in fact lands on the same MMI orbit, which 
consists of another $10$ facets. In total, $H_4$ thus has $20$ facets and is not simplicial anymore.

It is at $n=5$ that $H_n$ begins to exhibit a richer structure. The trivial zero-lift of SA now 
contributes an orbit of $15$ facets. The trivial zero-lift of MMI gives an orbit with $20$ facets.
There is now another inequivalent zero-lift of MMI which gives an orbit of length $45$. Besides these, 
there are $5$ orbits of genuinely new facets, given in Table \ref{tab:h5}.

In order, these give rise to orbits of lengths $72$, $90$, $10$, $60$ and $60$. 
Together with the $80$ facets coming from SA and MMI, there are a total of $372$ 
inequalities in the $H$-representation of $H_5$. Other than inequality \ref{eq:cyclic}, 
usually referred to as \textit{cyclic} due to its symmetry under $i\to i+1 \mod n$
which is manifest in the given representative, these inequalities are poorly understood.

\begin{table}[th!]
    \centering
    \hrule\vspace{.2cm}
    \begin{enumerate}[leftmargin=1.1cm,rightmargin=.6cm]
    \begin{minipage}[t]{\linewidth}   
        \item \label{eq:cyclic} $S_{123}+S_{234}+S_{345}+S_{145}+S_{125} 
        \ge S_{12}+S_{23}+S_{34}+S_{45}+S_{15}+S_{12345}$
    \end{minipage}~\\\vspace{.35cm}
    \begin{minipage}[t]{0.4\linewidth}   
        \item \label{eq:ineq52} $S_{14}+S_{23}+S_{125}+S_{135}+S_{145}+S_{245}+S_{345}
        \ge S_1+S_2+S_3+S_4+S_{15}+S_{45}+S_{235}+S_{1245}+S_{1345}$
    \end{minipage}\hspace{0.06\linewidth}
    \begin{minipage}[t]{0.56\linewidth}   
        \item \label{eq:ineq53} $S_{123}+S_{124}+S_{125}+S_{134}+S_{135}+S_{145}+S_{235}+S_{245}+S_{345}
        \ge S_{12}+S_{13}+S_{14}+S_{25}+S_{35}+S_{45}+S_{234}+S_{1235}+S_{1245}+S_{1345}$
    \end{minipage}~\\\vspace{.35cm}
    \begin{minipage}[t]{0.4\linewidth}
        \item \label{eq:ineq54} $2 S_{123}+S_{124}+S_{125}+S_{134}+S_{145}+S_{235}+S_{245} 
        \ge S_{12}+S_{13}+S_{14}+S_{23}+S_{25}+S_{45}+S_{1234}+S_{1235}+S_{1245}$
    \end{minipage}\hspace{0.06\linewidth}
    \begin{minipage}[t]{0.56\linewidth}   
        \item \label{eq:ineq55} $3 S_{123}+3 S_{124}+S_{125}+S_{134}+3S_{135}+S_{145}+S_{234}
        +S_{235}+S_{245}+S_{345} \ge 2 S_{12}+2 S_{13}+S_{14}+S_{15}+S_{23}+2 S_{24}+2 S_{35}
        +S_{45}+2 S_{1234}+2 S_{1235}+S_{1245}+S_{1345}$
    \end{minipage}
\end{enumerate}
\vspace{-.05cm}\hrule\vspace{.45cm}
    \caption{Representative inequalities in each of the $5$ new orbits of facets of $H_5$.}
    \label{tab:h5}
    \vskip -20pt
\end{table}

\subsection{Extreme rays and minimum graph realizations}
\label{app:rays}

Extreme rays and their minimum realizations in $K_{N_{min}}$ will be provided. Extreme rays will be 
labelled by a tuple $(n,N_{min}-n,\sigma)$, where $\sigma\ge 1$ is just an integer 
counting orbits at fixed $n$ and $N_{min}$ by listing their representatives lexicographically. 
Notice that $N_{min}-n\ge 1$ counts the number of bulk vertices needed in the 
minimum representation, plus the sink. For clarity, $S$-vector entries $S_I$ will be 
separated by a semicolon whenever the cardinality of $I$ increases.

At $n=1$ there is a single extreme ray with minimum realization the Bell pair in Figure \ref{fig:eg},
\begin{equation}
    S_{(1,1,1)} = (1).
\end{equation}
The $n=2$ cone has just the length-$3$ orbit of zero-lifts of the Bell-pair extreme ray $(1,1,1)$. 
For $n=3$, the Bell-pair zero-lift now gives an orbit of $6$ extreme rays. A new totally symmetric 
extreme ray appears. It has a star-graph minimum realization shown in Figure \ref{fig:eg} and reads
\begin{equation}
    S_{(3,2,1)} = (1,1,1;\;2,2,2;\;1).
\end{equation}
Lifting to $n=4$ we get orbits of $10$ extreme rays from $(1,1,1)$ and another $5$ from $(3,2,1)$.
A genuinely new length-$5$ orbit of extreme rays appears,
\begin{equation}
    S_{(4,2,1)} = (1,1,1,1;\;2,2,2,2,2,2;\;3,3,3,3;\;2),
\end{equation}
which again has a star graph as minimum realization, as shown in Figure \ref{fig:eg}.

\renewcommand{\thesubfigure}{}
\begin{figure}[H]
    \hfill
	\begin{subfigure}{.19\textwidth}
		\centering\includegraphics[width=\textwidth]{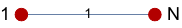}\caption{$(1,1,1)$}
	\end{subfigure}
	\hfill
	\begin{subfigure}{.16\textwidth}
		\centering\includegraphics[width=\textwidth]{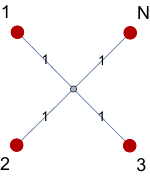}\caption{$(3,2,1)$}
	\end{subfigure}
	\hfill
	\begin{subfigure}{.19\textwidth}
		\centering\includegraphics[width=\textwidth]{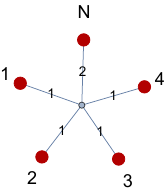}\caption{$(4,2,1)$}
	\end{subfigure}
	\hfill
	\medskip
	\caption[short]{Minimum realizations for extreme rays in each orbit of $H_n$ for $n \le 4$.}
	\vspace{-.0cm}
	\label{fig:eg}
\end{figure}

At $n=5$, lifted extreme rays become a minority. Extreme rays $(1,1,1)$, $(3,2,1)$ and $(4,2,1)$
respectively zero-lift to orbits of lengths $15$, $15$ and $30$, totaling just $60$ extreme rays. 
It turns out $H_5$ has $2267$ in total, so all the others are genuinely new ones. 
They fall into $16$ orbits, which we now present by increasing number of bulk vertices needed in their 
minimum realization.
There are $4$ distinct orbits of extreme rays realizable in a star graph,
\begin{equation}
\label{eq:exstar}
\begin{aligned}
    S_{(5,2,1)} &= (1,1,1,1,1;\;2,2,2,2,2,2,2,2,2,2;\;3,3,3,3,3,3,3,3,3,3;\;2,2,2,2,2;\;1),\\
    S_{(5,2,2)} &= (1,1,1,1,1;\;2,2,2,2,2,2,2,2,2,2;\;3,3,3,3,3,3,3,3,3,3;\;4,4,4,4,4;\;3),\\
    S_{(5,2,3)} &= (1,1,1,1,2;\;2,2,2,3,2,2,3,2,3,3;\;3,3,4,3,4,4,3,4,4,4;\;4,3,3,3,3;\;2),\\
    S_{(5,2,4)} &= (1,1,1,2,2;\;2,2,3,3,2,3,3,3,3,4;\;3,4,4,4,4,5,4,4,5,5;\;5,5,4,4,4;\;3),
\end{aligned}
\end{equation}
with respective orbit lengths $1$, $6$, $15$ and $60$. They can all be represented on the star 
graph shown in Figure \ref{f1}, with appropriate weight assignments as specified in Table \ref{tab:rays}.
There are $6$ orbits which require $2$ bulk vertices,
\begin{equation}
\begin{aligned}
    S_{(5,3,1)} &= (1,1,1,1,1;\;2,2,2,2,2,2,2,2,2,2;\;1,3,3,3,3,3,3,3,3,3;\;2,2,2,2,2;\;1),\\
    S_{(5,3,2)} &= (1,1,1,1,1;\;2,2,2,2,2,2,2,2,2,2;\;2,2,2,3,3,3,3,3,3,3;\;2,2,2,2,2;\;1),\\
    S_{(5,3,3)} &= (1,1,1,1,2;\;2,2,2,3,2,2,3,2,3,3;\;3,3,2,3,4,4,3,4,4,4;\;4,3,3,3,3;\;2),\\
    S_{(5,3,4)} &= (1,1,2,2,2;\;2,3,3,3,3,3,3,4,4,4;\;4,4,4,3,3,5,5,5,5,4;\;4,4,4,3,3;\;2),\\
    S_{(5,3,5)} &= (2,2,2,2,3;\;4,4,4,5,4,4,5,4,5,5;\;4,6,5,6,5,7,6,7,7,7;\;6,5,5,5,5;\;3),\\
    S_{(5,3,6)} &= (3,3,3,3,3;\;6,6,6,6,6,6,6,6,6,6;\;5,7,7,7,7,9,9,9,9,9;\;6,6,6,6,6;\;3),
\end{aligned}
\end{equation}
with respective orbit lengths $10$, $60$, $90$, $180$, $360$ and $90$.
These can be represented on graphs in Figures \ref{f2} to \ref{f5} following Table \ref{tab:rays}.
There are $6$ orbits which require $3$ bulk vertices,
\begin{equation}
\begin{aligned}
    S_{(5,4,1)} &= (1,1,1,1,1;\;2,2,2,2,2,2,2,2,2,2;\;2,2,2,2,3,3,3,3,3,3;\;2,2,2,2,2;\;1),\\
    S_{(5,4,2)} &= (1,1,1,1,1;\;2,2,2,2,2,2,2,2,2,2;\;2,2,3,3,2,3,3,3,3,3;\;2,2,2,2,2;\;1),\\
    S_{(5,4,3)} &= (2,2,2,2,3;\;4,4,4,5,4,4,5,4,5,5;\;4,6,5,6,7,5,6,7,7,7;\;6,5,5,5,5;\;3),\\
    S_{(5,4,4)} &= (3,3,3,3,3;\;6,6,6,6,6,6,6,6,6,6;\;5,7,7,7,9,9,9,7,9,9;\;6,6,6,6,6;\;3).
\end{aligned}
\end{equation}
with respective orbit lengths $180$, $60$, $360$ and $360$.
These are realizable in graphs in Figures \ref{f6} to \ref{f8} following Table \ref{tab:rays}.

\noindent Finally, there is an orbit of length $360$ with $4$ bulk vertices,
\begin{equation}
    S_{(5,5,1)} = (3,3,3,3,3;\;6,6,6,6,6,6,6,6,6,6;\;5,7,7,7,9,7,9,9,9,9;\;6,6,6,6,6;\;3),
\end{equation}
and another one of length $15$ with $5$ bulk vertices,
\begin{equation}
    S_{(5,6,1)} = (1,1,1,1,1;\;2,2,2,2,2,2,2,2,2,2;\;2,2,3,3,2,2,3,3,3,3;\;2,2,2,2,2;\;1).
\end{equation}
These are respectively realizable in graphs in Figures \ref{f7} and \ref{f8} following Table \ref{tab:rays}.

In summary, $H_5$ consists of $2267$ extreme rays in $19$ orbits, 
$2207$ of which lie in $16$ orbits new to $n=5$.
Note that apart from the Bell pair $(1,1,1)$ in Figure \ref{fig:eg}, there are no edges
between terminals in any of the minimum extreme-ray representations.
Each is planar except for $(5,6,1)$ in Figure \ref{f8}, 
which can be embedded on a torus. Terminals have degree at most $3$.

\section{Miscellaneous examples}
\label{newS}

\paragraph{Convexity of $H_{N,n}^+$ and $H_{N,n}$:}

For $n=5$, consider extreme rays $S_{(5,2,2)}$ and the zero-lift of $S_{(3,2,1)}$, both of which 
are realizable in $K_7$. Using the ILP method in Section \ref{ilp} we can determine that their sum 
is not. Therefore, neither $H_{7,5}^+$ nor $H_{7,5}$ is convex without the convex operator applied.
The minimum realization of their sum is in $K_8$, and has edges 
$\{(1, 6), (2, 6), (3, 6), (4, 7), (5, 7), (6, 7), (7, 8)\}$ 
with respective weights $\{2,2,2,1,1,4,4\}$.

\paragraph{Extreme rays of $P_{N,n}$ and $H_{N,n}$:}
For $n=5$ in $K_9$, consider these three feasible $(S,w)$ pairs:
\begin{equation*}
\begin{aligned}
    S^1&= (3,4,3,3,3;\;7,6,6,6,7,7,7,6,6,6;\;6,8,6,9,9,7,8,8,8,9;\;5,5,5,6,5;\;2), \\
    S^2&= (1,2,1,1,1;\;3,2,2,2,3,3,3,2,2,2;\;2,2,2,3,3,3,2,2,2,3;\;1,1,1,2,1;\;0), \\
    S^3&= (2,2,2,2,2;\;4,4,4,4,4,4,4,4,4,4;\;4,6,4,6,6,4,6,6,6,6;\;4,4,4,4,4;\;2), \\
    w^1&= (0,0,0,0,2,1,0,0;\;0,0,0,3,1,0,0;\;0,0,2,0,1,0;\;0,1,1,1,0;\;1,2,0,0;\;0,1,0;\;1,0;\;2), \\
    w^2&= (0,0,0,0,1,0,0,0;\;0,0,0,2,0,0,0;\;0,0,1,0,0,0;\;0,1,0,0,0;\;1,0,0,0;\;0,0,0;\;0,0;\;0), \\
    w^3&= (0,0,0,0,1,0,1,0;\;0,0,0,0,0,2,0;\;0,0,0,1,1,0;\;0,1,1,0,0;\;2,0,0,0;\;1,1,0;\;1,2;\;0).
\end{aligned}
\end{equation*}
Here $(S^1,w^1)$ is an extreme ray of $P_{9,5}$. However, its $S$ coordinates have $S^1=S^2+S^3$,
so the latter cannot be an extreme ray of $H_{9,5}$ and hence of $H_5$ either
(cf. Theorem \ref{thm1}\ref{thm:HP}).

\vfill

\begin{table}[bh!]
\setlength{\tabcolsep}{1pt}
\renewcommand{\arraystretch}{.92}
    \small
    \centering
    \begin{tabular}{c | c | cccccc | cccccccccccccccc}
    \toprule
        \multicolumn{1}{c|}{~Extreme Ray~~}
     &  \multicolumn{1}{c|}{~Graph~~}
     &  \multicolumn{6}{c|}{~Terminal Edges~} 
     &  \multicolumn{15}{c}{Edge Weights} \\
        \cmidrule(lr){1-1}
        \cmidrule(lr){2-2}
        \cmidrule(lr){3-8}
        \cmidrule(lr){9-24}     
        $(\,\cdot\,,\,\cdot\,,\,\cdot\,)$ & $\#$ & ~$N$ & $1$ & $2$ & $3$ & $4$ & $5$~ & 
        $w_{N}$ & $w_{1}$ & $w_{2}$ & $w_{3}$ & $w_{4}$ & $w_{5}$ & $w_{6}$ & $w_{7}$ & $w_{8}$ & $w_{9}$ & 
        $w_{10}$ & $w_{11}$ & $w_{12}$ & $w_{13}$ & $w_{14}$ & $w_{15}$ \\
    \midrule\midrule
    $(5,2,1)$ & $1$ & & & & & & & $ 1$ & $ 1$ & $ 1$ &
    $ 1$ & $ 1$ & $ 1$ & & & & & & & & & & \\
    $(5,2,2)$ & $1$ & & & & & & & $ 3$ & $ 1$ & $ 1$ &
    $ 1$ & $ 1$ & $ 1$ & & & & & & & & & & \\
    $(5,2,3)$ & $1$ & & & & & & & $ 2$ & $ 1$ & $ 1$ &
    $ 1$ & $ 1$ & $ 2$ & & & & & & & & & & \\
    $(5,2,4)$ & $1$ & & & & & & & $ 3$ & $ 1$ & $ 1$ &
    $ 1$ & $ 2$ & $ 2$ & & & & & & & & & & \\\midrule
    $(5,3,1)$ & $2$ & & & & & & & $ 1$ & $ 1$ & $ 1$ &
    $ 1$ & $ 1$ & $ 1$ & $ 1$ & & & & & & & & & \\
    $(5,3,2)$ & $3$ & & & & & & & $ 2$ & $ 2$ & $ 2$ &
    $ 1$ & $ 1$ & $ 1$ & $ 1$ & $ 1$ & $ 1$ & $ 1$ & & & & & & \\
    $(5,3,3)$ & $2$ & & & & $ w_5$ & & $ w_3$~ & $ 2$ & $ 1$ & $ 1$ &
    $ 2$ & $ 1$ & $ 1$ & $ 2$ & & & & & & & & & \\
    $(5,3,4)$ & $4$ & & & & & & & $ 2$ & $ 1$ & $ 1$ &
    $ 2$ & $ 1$ & $ 1$ & $ 1$ & $ 1$ & $ 1$ & & & & & & & \\
    $(5,3,5)$ & $3$ & ~$ w_1$ & $ w_N$ & $ w_4$ & & $ w_2$ & & $ 2$ & $ 3$ & $ 2$ &
    $ 1$ & $ 1$ & $ 2$ & $ 1$ & $ 1$ & $ 1$ & $ 1$ & & & & & & \\
    $(5,3,6)$ & $5$ & & & & & & & $ 3$ & $ 3$ & $ 2$ &
    $ 2$ & $ 2$ & $ 2$ & $ 1$ & $ 1$ & $ 1$ & $ 1$ & $ 1$ & & & & & \\\midrule
    $(5,4,1)$ & $6$ & & & & & & & $ 2$ & $ 1$ & $ 2$ &
    $ 1$ & $ 1$ & $ 1$ & $ 1$ & $ 1$ & $ 1$ & $ 1$ & $ 1$ & $ 1$ & & & & \\
    $(5,4,2)$ & $7$ & & & & & & & $ 1$ & $ 2$ & $ 1$ &
    $ 1$ & $ 2$ & $ 2$ & $ 1$ & $ 1$ & $ 1$ & $ 1$ & $ 1$ & $ 1$ & & & & \\
    $(5,4,3)$ & $7$ & ~$ w_1$ & $ w_N$ & $ w_3$ & $ w_4$ & $ w_2$ & & $ 1$ & $ 3$ & $ 1$ &
    $ 1$ & $ 2$ & $ 3$ & $ 2$ & $ 1$ & $ 1$ & $ 1$ & $ 1$ & $ 1$ & & & & \\
    $(5,4,4)$ & $8$ & & & & & & & $ 3$ & $ 2$ & $ 2$ &
    $ 1$ & $ 2$ & $ 2$ & $ 1$ & $ 1$ & $ 1$ & $ 1$ & $ 1$ & $ 1$ & $ 1$ & $ 1$ & & \\\midrule
    $(5,5,1)$ & $9$ & & & & & & & $ 3$ & $ 3$ & $ 3$ &
    $ 1$ & $ 1$ & $ 3$ & $ 2$ & $ 2$ & $ 1$ & $ 1$ & $ 1$ & $ 1$ & $ 1$ & $ 1$ & $ 1$ & \\\midrule
    $(5,6,1)$ & $10$ & & & & & & & $ 1$ & $ 1$ & $ 1$
    & $ 1$ & $ 1$ & $ 1$ & $ 1$ & $ 1$ & $ 1$ & $ 1$ & $ 1$ & $ 1$ & $ 1$ & $ 1$ & $ 1$ & $ 1$ \\
    \bottomrule
    \end{tabular}
    \caption{Extreme ray vs graph cross reference table for Figure \ref{extremegraphs}. 
    The columns under ``Terminal Edges'' specify which edge $w_j$ is incident to each 
    extended terminal $i\in[5;N]$, with a blank entry indicating $j = i$. Edge weights give all 
    extreme rays listed throughout Appendix \ref{app:rays}, up to overall scaling.}
    \label{tab:rays}
\end{table}

\newpage

\renewcommand{\thesubfigure}{\itshape{(\arabic{subfigure}})}
\begin{figure}[ht!]
	\begin{subfigure}{.18\textwidth}
		\centering\includegraphics[width=\textwidth]{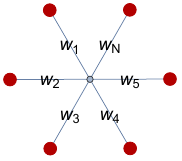}\caption{$N=7$}
		\label{f1}
	\end{subfigure}\hfill
	\begin{subfigure}{.26\textwidth}
		\centering\includegraphics[width=\textwidth]{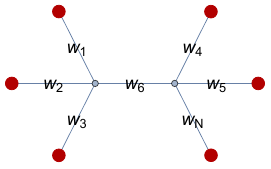}\caption{$N=8$}
		\label{f2}
	\end{subfigure}\hfill
	\begin{subfigure}{.29\textwidth}
		\centering\includegraphics[width=\textwidth]{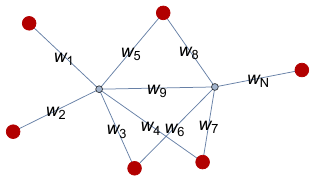}\caption{$N=8$}
		\label{f3}
	\end{subfigure}\hfill
	\begin{subfigure}{.25\textwidth}
		\centering\includegraphics[width=\textwidth]{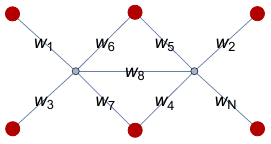}\caption{$N=8$}
		\label{f4}
	\end{subfigure}\\[10pt]
	\begin{subfigure}{.28\textwidth}
		\centering\includegraphics[width=\textwidth]{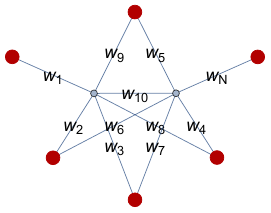}\caption{$N=8$}
		\label{f5}
	\end{subfigure}\hfill
	\begin{subfigure}{.32\textwidth}
		\centering\includegraphics[width=\textwidth]{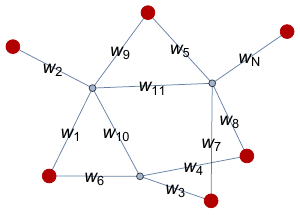}\caption{$N=9$}
		\label{f6}
	\end{subfigure}\hfill
	\begin{subfigure}{.25\textwidth}
		\centering\includegraphics[width=\textwidth]{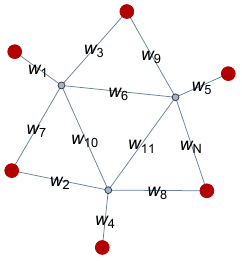}\caption{$N=9$}
		\label{f7}
	\end{subfigure}\\[12pt]
	\begin{subfigure}{.34\textwidth}
		\centering\includegraphics[width=\textwidth]{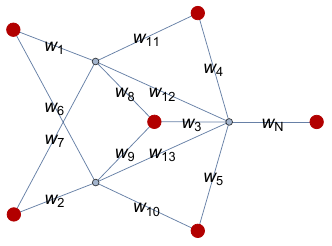}\caption{$N=9$}
		\label{f8}
	\end{subfigure}\hfill
	\begin{subfigure}{.32\textwidth}
		\centering\includegraphics[width=\textwidth]{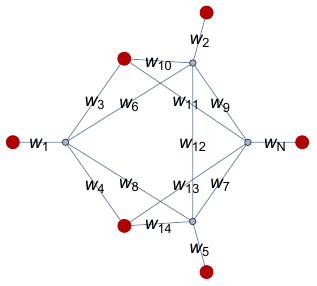}\caption{$N=10$}
		\label{f9}
	\end{subfigure}\hfill
	\begin{subfigure}{.26\textwidth}
		\centering\includegraphics[width=\textwidth]{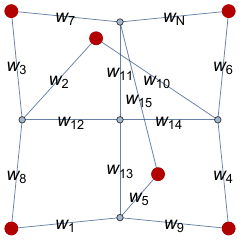}\caption{$N=11$}
		\label{f10}
	\end{subfigure}
	\vspace{14pt}
	\caption[short]{Minimum realizations for extreme rays which are new for $H_5$.}
	\label{extremegraphs}
\end{figure}

\end{document}